\documentclass[11pt]{article}
\usepackage[utf8]{inputenc}
\usepackage{float}
\usepackage{multirow}
\usepackage{color}
\usepackage{amsmath,amssymb,amsthm,amsfonts,amstext,amsbsy,amscd}
\usepackage[a4paper, left=2.5cm, right=2.5cm, top=3cm, bottom=3cm]{geometry}
\usepackage{mathrsfs}
\usepackage{graphics}
\usepackage{graphicx}
\usepackage{float}
\usepackage{multirow}
\usepackage{multicol}
\usepackage{latexsym}

\usepackage{rotating}
\providecommand{\Fourier}{\mathcal{F}}
\usepackage{diagbox}

\usepackage{color}
\RequirePackage[colorlinks,citecolor=blue,urlcolor=blue,linkcolor=blue]{hyperref}

\definecolor{forestgreen}{rgb}{0.13, 0.55, 0.13}

\renewcommand{\phi}{\varphi}

\newcommand{\PP}{\mathbb{P}}

\newcommand{\E}{\mathbb{E}}

\newcommand{\R}{\mathbb{R}}
\newcommand{\N}{\mathbb{N}}
\newcommand{\V}{\mathbb{V}}

\numberwithin{equation}{section}
\newtheorem{theorem}{Theorem}[section]
\newtheorem{lemma}{Lemma}[section]
\newtheorem{definition}{Definition}[section]

\newtheorem{remark}{Remark}[section]
\newtheorem{example}{Example}[section]

\renewcommand{\tilde}{\widetilde}
\DeclareMathOperator{\lk}{L}
\renewcommand{\hat}{\widehat}

\newcommand{\Cov}{\textnormal{Cov}\, }

\providecommand{\eps}{\varepsilon}

\RequirePackage[OT1]{fontenc}
\RequirePackage{amsthm,amsmath}
\RequirePackage[numbers]{natbib}
\RequirePackage[colorlinks,citecolor=blue,urlcolor=blue]{hyperref}

\date{}
\begin{document}

\title{Adaptive directional estimator of the density in ${\mathbb R}^d$ for independent and mixing sequences}
\maketitle

\author{Sinda Ammous \footnote{Universit\'e Paris Cit\'e, CNRS,  UMR 8145, MAP5,
F-75006 Paris, France}, 
Jérôme Dedecker \footnote{Universit\'e Paris Cit\'e, CNRS,  UMR 8145, MAP5,
F-75006 Paris, France}, 
Céline Duval\footnote{Universit\'e de Lille, CNRS, UMR 8524 - Laboratoire Paul Painlev\'e, F-59000 Lille, France}
}
\begin{abstract} A new multivariate density estimator for  stationary sequences is obtained by  Fourier inversion of the thresholded empirical characteristic function. This estimator  does not depend on the choice of  parameters related to the smoothness of the density; it is directly adaptive. We establish oracle inequalities  valid for independent, $\alpha$-mixing and $\tau$-mixing sequences, which allows us to  derive optimal convergence rates, up to a logarithmic loss.   On general anisotropic Sobolev classes,  the estimator adapts to the regularity of the unknown density but also achieves directional adaptivity. In particular, if $A$ is an invertible matrix,  if the observations are drawn from $X\in\R^{d}$, $d\ge 1$, it achieves the rate implied by the regularity of  $AX$, which may be more regular than $X$. The estimator is easy to implement and numerically efficient. It depends on the calibration of a parameter for which we propose an innovative numerical selection procedure, using the Euler characteristic of the thresholded areas.
\end{abstract}


\section{Introduction}

The problem of adaptive density estimation   was mainly studied in a context of i.i.d. (independent and identically distributed) random variables,  but also for dependent data since the late 90's. Given the considerable amount of articles on the subject, the following review does not pretend to be an exhaustive rendition of all the works on the topic. On minimax adaptive multivariate density estimation from direct independent  random vectors, we mention projection methods, with the preliminary work of Hasminskii and Ibragimov (1990) \cite{10.1214/aos/1176347736} or more recently Birgé (2014)  \cite{birge2014model}, wavelet techniques with Kerkyacharian and Picard (2000)  \cite{MR1821645} and kernel estimators investigated e.g. in Chac\'on and Duong (2010) \cite{chacon2010multivariate}, the works of Goldenshluger and Lepski   \cite{goldenshluger2011bandwidth,goldenshluger2014adaptive} or Rebelles (2015) \cite{10.1214/15-EJS986}. 

Let us look in more detail at  two recent papers. In \cite{comte2013anisotropic}, Comte and Lacour  consider the problem of deconvolution in ${\mathbb R}^d$, but their results also apply to the case where the variables are observed directly.  The authors propose adaptive density estimators for both the pointwise $\lk^2$-risk and the integrated $\lk^2$-risk, by adapting the method of Goldenshluger and Lepski. For the integrated $\lk^2$-risk, they obtain adaptive estimators on general anisotropic Sobolev classes, using tensor products of the sinus-cardinal (sinc) kernel. In \cite{lacour2017estimator}, Lacour {\it et al.} propose a new method for selecting the bandwidth of kernel density estimators. For a class of kernels described in 
\cite{goldenshluger2014adaptive} they obtain adaptive estimators for anisotropic Nikolskii classes and the integrated $\lk^2$-risk. 

 Regarding dependent sequences, 
 The only article that we know of that deals with adaptive multivariate density estimation in a mixing context, is that of Bertin {\it et al.} \cite{BKLP20} . These authors deal with density estimation on a bounded domain of ${\mathbb R}^d$, for geometrically $\beta$-mixing sequences (see Volkonski and Rozanov(1959) \cite{volkonskii1959some} for the  definition of $\beta$-mixing coefficients).
 
The other articles we  briefly describe below only concern the estimation of the density in the univariate case.
 Tribouley and Viennet (1998) \cite{tribouley1998lp} propose $\lk^{p}$ adaptive estimators for the common density $f$ of a stationary $\beta$-mixing sequence   using wavelets, that are minimax optimal (in the sense that the procedure leads to the same minimax rate as in the i.i.d. case). Considering discrete or continuous time stationary processes, Comte and Merlevède (2002) \cite{comte2002adaptive}  study the  adaptive density estimation of the common density $f$ of a $\alpha$-mixing or $\beta$-mixing process (see  Definition \ref{def alpha} below for the definition of $\alpha$-mixing coefficients in the sense of Rosenblatt  (1956) \cite{rosenblatt1956central}). 
They use  a penalized least square method  to compute the adaptive estimators. Comte {\it et al.} (2008) \cite{CDT2008} propose a model selection procedure for projection estimators on the Shannon basis to estimate $f$ on the whole real line for stationary $\beta$-mixing sequences.
Lerasle  \cite{lerasle2009adaptive,lerasle2011optimal} proposes an adaptive estimator based on model selection for the  density of a stationary process  which is either $\beta$-mixing or $\tau$-mixing (see Definition \ref{def tau} below for the definition of $\tau$-mixing coefficients).
Gannaz and Wintenberger (2010) \cite{gannaz2010adaptive} consider other type of dependence; they estimate the common marginal density $f$ by a wavelet type estimator, under some conditions on the  $\lambda$ or $\tilde{\phi}$  coefficients (see \cite{dedecker2007weak} for  their definition).  
Asin and Johannes (2017) \cite{asin2017adaptive}  give a data-driven non-parametric estimation procedure for a density and a regression function. They use an orthogonal series estimator for  $\beta$-mixing sequences. Bertin and Klutchnikoff (2017) \cite{bertin2017pointwise}  estimate the common  density $f$ of a weakly dependent process, giving upper bounds for the pointwise risk, and  propose a data-driven procedure based on Goldenshluger and Lepskii \cite{goldenshluger2011bandwidth} for $\alpha$-mixing and $\lambda$-dependent sequences.

We observe that many papers on adaptive density estimation in a dependent context deal with $\beta$-mixing sequences. The reason is a technical one: for $\beta$-mixing sequences, one can use Berbee's coupling lemma \cite{Berbee79}, which enables to go back to the i.i.d. case and to apply the powerful concentration inequalities of Talagrand  \cite{talagrand1996new}. This approach is much trickier in the $\alpha$-mixing case 
because the coupling tools for $\alpha$-mixing sequences are much less efficient (see for instance \cite{rio2017asymptotic} Chapter 5). 
$\tau$-mixing sequences have better coupling properties (see  \cite{dedecker2005new}), but they are quite complicated to handle because of their hereditary properties (individual functions of $\tau$-mixing sequences are not necessarily $\tau$-mixing).

Let us now briefly underline why the $\alpha$-mixing case is so attractive. It is well known that the notion of $\alpha$-mixing is the weakest type of mixing among the usual mixing conditions (see e.g. Bradley \cite{Bradley86}). It contains two large classes of examples: irreducible, aperiodic and positively recurrent Markov chains (for which $\beta$-mixing is equivalent to $\alpha$-mixing, see \cite{rio2017asymptotic}, Chapter 9), as well as functions of Gaussian processes which are naturally $\rho$-mixing (see \cite{Bradley86} for the definition of $\rho$-mixing). 

However, $\alpha$-mixing has some limitations:  it does not apply in general to non-irreducible Markov chains. For instance the Markov chains associated with most dynamical systems are not $\alpha$-mixing (see for instance \cite{dedecker2005new}). By contrast, the $\tau$-mixing coefficients can be computed for  a large class of dynamical systems, as well as for many other examples (such as functions of linear processes with discrete innovations).
Of course, other notions of dependence can also be defined: we refer to \cite{dedecker2005new} and   \cite{dedecker2007weak} for an overview.

The  purpose of this paper is to introduce a unified adaptive density estimator for multivariate stationary random variables that are independent, $\alpha$-mixing  or $\tau$-mixing (see Equation \eqref{eq:estfbis} below for the definition of the estimator). It is inspired by the recent adaptive procedure introduced by Duval and Kappus (2019) \cite{duval2019adaptive}  to select the optimal cutoff parameter for univariate density estimators based on a Fourier method, in the i.i.d. case. This procedure is numerically easy to implement and the concentration tool used to establish the oracle inequality is a basic Hoeffding inequality for partial sum of bounded random variables, which makes it eligible for an extension to the case  of dependent variables. More precisely, we shall use  Fuk-Nagev type inequalities,  that can be derived from  \cite{rio2017asymptotic} in the $\alpha$-mixing case (see Lemma \ref{lem:FG}), and from  \cite{DP04} in the $\tau$-mixing case (see Lemma \ref{lem:FGtau}). 
Equation \eqref{eq:estfbis} below proposes a complete rewriting of the estimator of \cite{duval2019adaptive} for multivariate random variables.

The contributions of the present work are the following.
We propose a unique thresholded estimation procedure which is  parameter free. Indeed, if the procedure depends on a thresholding constant  (unrelated to the smoothness of the unknown density), this constant can be easily calibrated  through an innovative tool using the Euler characteristic of thresholded area.  Oracle inequalities are established for independent and dependent sequences of multivariate random variables (which seems quite new in a dependent context, in particular regarding the adaptivity to anisotropic regularity classes).  The proofs of the oracle inequalities rely on fine cuttings of the quadratic risk, which makes it elegant and easy to adapt in other contexts such as indirect observations. The bias term appearing in these inequalities has such a  general shape that it allows anisotropic regularity classes with a possible  change of base (not necessarily orthonormal) that may lead to faster convergence rates. A property that, to the knowledge of the authors, has not been studied in the literature and is not shared with existing estimators.  We call  this property ``Directional Adaptivity''.\\

The paper is organized as follows. Section \ref{sec:iid} is devoted to the multivariate  independent case;  we first define our adaptive estimator and state Theorem \ref{thm:AD} which is an oracle inequality. Some extensions are discussed as well as a comparison with wavelet thresholded estimators. Rates of convergence are derived on general anisotropic Sobolev classes, where a linear bijective transformation of the data is allowed (see Equation \eqref{eq:Sob}). As emphasized by Example \ref{ex:Ref} this transformation can improve the rates of convergence. The remaining of the Section is devoted to its numerical study. In Section \ref{sec:kappa} we propose a numerical procedure for the calibration of the only  parameter appearing in the definition of the estimator, using the Euler characteristic of the thresholded areas. Section \ref{sec:simiid} provides  numerical illustrations in dimension 2. Section \ref{sec:dep} generalizes Theorem \ref{thm:AD} for stationary sequences that are either $\alpha$-mixing (see Definition \ref{def alpha} and Theorem \ref{thm:ADalpha}) or $\tau$-mixing (see Definition \ref{def tau} and Theorem \ref{thm:ADtau}). In each case the numerical performances of the estimator, combined with the selection procedure of Section \ref{sec:kappa}, are investigated in dimension 1. All the proofs are  gathered in Section \ref{sec:prf}.

\paragraph{Notations}
Given a random variable $Z$ taking values in $\R^{d}$,
$\phi_Z(u)=\E[e^{i\langle u, Z\rangle} ]
$ denotes the characteristic function of $Z$. For $f\in \lk^1(\R^{d})$, $$\Fourier f(u) =\int _{\R^{d}}
e^{i\langle u,x\rangle } f(x) {\rm d}x, \quad u\in\R^{d},$$
 is understood to be the Fourier transform of $f$.  Moreover, we denote by $\| \cdot  \|$ the $\lk^2$-norm of functions, $\| f\|^2:= \int _{\R^{d}}|f(x)|^2 {\rm d}x$.  

Let ${\mathcal A}$ be the class of invertible $d\times d$ matrices such that $A([-1,1]^d) \subseteq [-1,1]^d$. Note that, if $A \in {\mathcal A}$, then $|\text{det}(A)|\leq 1$. A necessary and sufficient condition for an invertible matrix $A$ to belong to ${\mathcal A}$ is that the $\ell_1$ norm of each row is less than 1.
Given $f\in \lk^1(\R^{d}) \cap \lk^2(\R^{d})$ and  $A \in {\mathcal A}$, we denote by $f_{A,m}$, $m\in\R^{d}_{+}$, the  uniquely defined function with Fourier transform $\Fourier f_{A,m}= (\Fourier f)\mathbf{1}_{A([-m,m])}$, where $[-m,m]=[-m_{1},m_{1}]\times\cdots\times[-m_{d},m_{d}]$.  The interest of introducing the matrix $A$ is to capture directional adaptivity and consider general anisotropy classes as defined in \eqref{eq:Sob} (see also Example \ref{ex:Ref}).

\section{Directional adaptive procedure: the independent case \label{sec:iid}}

\subsection{Adaptive  thresholded estimator}

Consider $n$ i.i.d. realisations $X_j,\, 1\leq j\leq n,$ of a ${\mathbb R}^d$-valued random variable $X$, with Lebesgue-density $f$.  We build an estimator $\hat\phi_{X,n}$ of  $\phi_{X}$ as follows\[\hat\phi_{X,n}(u)=\frac{1}{n}\sum_{j=1}^{n}e^{i\langle u,X_{j}\rangle},\ u\in\R^{d}.\]
Using Fourier inversion,  we define  an estimator of $f$ as follows
\begin{align}\label{eq:estfbis}\hat{f}_n(x)  = \frac{1}{(2\pi)^d} \int_{[-n,n]^{d}}e^{- i\langle u,x\rangle} \tilde{\phi}_{X,n}(u) {\rm d}u,\quad x\in\R^{d}. 
\end{align}
where for  $\kappa_{n}:=(1+\kappa\sqrt{\log n})$, for some positive $\kappa$, we set 
\begin{align}\label{eq:phiYcut}
\tilde{\phi}_{X,n}(u)=\hat{\phi}_{X,n}(u) \mathbf{1}_{\{|\hat{\phi}_{X,n}(u)|\geq {\kappa}_n n^{-1/2}\}},\quad u\in\R^{d}.
\end{align} The quantity in \eqref{eq:estfbis} has no reason to be positive therefore in practice we take $\hat{f}_n=\big($Re$(\hat{f}_n)\big)_{+}$. We underline that contrary to classical Fourier estimators the cutoff parameter is taken equal to $n$ and that adaptivity will be obtained by thresholding the estimated characteristic function. Therefore the only parameter in \eqref{eq:estfbis} that requires a choice is the constant $\kappa$ appearing in the threshold. This quantity plays a role only in the order of remainder terms (see Theorem \ref{thm:AD}). This is why we choose not to make explicit the dependency in $\kappa$ of $\hat f_{n}$ as the choice of $\kappa$ is not related to the smoothness of $f$.    In Section \ref{sec:kappa}, we propose an innovating adaptive procedure to select $\kappa$ using the Euler characteristic of the set $\{u, |\hat{\phi}_{X,n}(u)|\geq {\kappa}_n n^{-1/2}\}$, which performs quite well numerically.

\begin{theorem}\label{thm:AD}
Let $ \kappa > 0$. The following inequality holds
\begin{multline*}
\E[  \| \hat{f}_n- f\|^2 ]  \leq  \underset{m\in[0,n]^{d}, A \in {\mathcal A}}{\inf}\left( 18
\|f-f_{A,m}\|^{2}
+\frac{\left (10+2(1+(\kappa+2) \sqrt{\log n})^{2}\right )|\mathrm{det}(A)|}{\pi^d n}  m_{1}\cdots m_{d}  \right)\\+  \frac{64}{\pi^d}n^{d - \kappa^2/4}.    
\end{multline*}
\end{theorem}

Note that a choice of $\kappa>2\sqrt{d+1}$ ensures that the last term $n^{d - \kappa^2/4}$ is negligible (see Remark \ref{rem:hu} and \ref{rem:Dn} below for modifications of the estimator $\hat{f}_n$ allowing a choice of $\kappa$ independent of $d$). Proof of Theorem \ref{thm:AD} is given in Section \ref{sec:prfiid}, it is self contained and relies on fine cuttings of the quadratic risk. The inequality involved is an Hoeffding inequality, which makes the proof robust to other contexts such as $\alpha$ or $\tau$ mixing sequences as shown in Section \ref{sec:dep}.  A discussion on the resulting rates of convergence is given in the next Section. 
The generalization to indirect measurements such as a deconvolution framework with known noise distribution could be easily derived (see  \cite{duval2019adaptive})  by replacing $\phi_{X}$ by $\phi_{Y}/\phi_{\eps}$ if one observes i.i.d. realizations of $Y = X+\eps$ where $X$ is independent of $\eps$. 

The spirit of estimator \eqref{eq:estfbis} is closely related to wavelet thresholded estimator where adap\-ti\-vi\-ty is achieved by thresholding empirical wavelet coefficients that are too small (namely smaller than a constant times $\sqrt{\log n/n}$) to suppress noise artefacts (see e.g. Kerkyacharian and Picard (2000) \cite{MR1821645} Sections 5 and 6 (in particular Theorem 6.1) for the multivariate isotropic density estimation (see also Donoho et al. (1995) \cite{donoho1995wavelet}) or Neumann (2000)  \cite{MR1769750} (in particular Theorem 2.3) for anisotropic estimation in a Gaussian white noise model). As for estimator \eqref{eq:estfbis}, thres\-hol\-ding enables to directly define an adaptive minimax estimator which requires no calibration of a parameter depending on the unknown smoothness of the function, at the cost of a logarithmic loss. If wavelets method have the advantage of facilitating the control of $\lk^{p}$ loss functions, the proof of Theorem \ref{thm:AD} relies on direct arguments allowing the generalization to mixing sequences as shown in Section \ref{sec:dep}.

\begin{remark}\label{rem:hu}
The remaining term $4n^{d-\kappa^{2}/4}$ in Theorem \ref{thm:AD} can be improved in $C_{\kappa,d}n^{-\kappa^{2}/4}$  by making the threshold in \eqref{eq:phiYcut}  depend on $u$. For a function $h:\R^{d}\mapsto \R_{+}$, let $\check{\phi}_{X,n}(u)=\hat{\phi}_{X,n}(u) \mathbf{1}${\scriptsize ${\{|\hat{\phi}_{X,n}(u)|\geq {\kappa}\sqrt{\log(h(u)n)/n} \}}$}. It is then straightforward to replace the bound in Lemma \ref{lem:Hoeffding} below by $4 (nh(u))^{-\kappa^{2}/{4}}$. Depending on the choice of $h$ the remaining term of Theorem \ref{thm:AD} changes (see Inequality \eqref{eq:VtoBH}). For $h\equiv 1$  we recover the remaining term  $4n^{d-\kappa^{2}/4}$;  choosing $h(u)=(1+|u_{1}|)\ldots(1+| u_{d}|)$ enables to replace this term by $C_{\kappa,d}n^{-\kappa^{2}/4}$ which is negligible for $\kappa>2$. \end{remark}

\begin{remark}\label{rem:Dn}
If one does not look for the directional  adaptivity given by the infimum over the class ${\mathcal A}$,  there is another way to get rid of the quantity $n^d$ appearing in the remaining term of Theorem \ref{thm:AD}.  It suffices  to modify the estimator $\hat{f}_n$ as follows : let 
$$
D_n=[-n, n]^d \cap \{ u \in {\mathbb R}^d : |u_1 \cdots u_d | \leq n \}\, 
\quad 
\text{and }
\quad
\text{\v{f} }(x)  = \frac{1}{(2\pi)^d} \int_{D_n}e^{- i\langle u,x\rangle} \tilde{\phi}_{X,n}(u) {\rm d}u \, .
$$
Following exactly the proof of Theorem \ref{thm:AD} and letting $f_m=f_{Id,m}$, one gets the oracle inequality:
\begin{multline*}
\E[  \| \text{\v{f}}- f\|^2 ]  \leq  \underset{m\in[0,n]^{d}, m_1\cdots m_d \leq n}{\inf}\left( 18
\|f-f_{m}\|^{2}
+\frac{10+2(1+(\kappa+2) \sqrt{\log n})^{2}}{\pi^d n}m_{1}\cdots m_{d}  \right)\\
+  C_{d} n^{1-\kappa^{2}/4}(\log n)^{d-1},    
\end{multline*}
where $C_{d}$ is a constant depending only on $d$. Here, the remaining term in the oracle inequality comes from the estimate
\begin{align}
\label{eq:VolDn}\int_{D_{n}}{\rm d}u = \int_{[-n,n]^{d}}\mathbf{1}_{|u_{1}\cdots u_{d}|\le n}{\rm d}u_{1}\ldots{\rm d}u_{d}\underset{n\to\infty}{\sim }  \frac{2^{d}(d-1)^{d-1}}{(d-1)!}\log(n)^{d-1}n.
\end{align}
 Proof of Equation \eqref{eq:VolDn} is given in  Section \ref{sec:prf}.
\end{remark}

\subsection{Rates of convergence on general anisotropic Sobolev classes}
To derive the resulting rate of convergence we require some regularity on the density $f$ to control the order of the bias term.  Theorem \ref{thm:AD} implies that our adaptive estimator $\hat{f}_n$ is rate optimal on Sobolev regularity classes, with a regularity direction given by any matrix $A \in {\mathcal A}$, up to a logarithmic loss. 
Indeed, consider a Sobolev class $\mathcal{S}(A,\mathbf s,L)$ for $\mathbf s=(s_{1},\ldots,s_{d})$, $L>0$  and $A \in {\mathcal A}$, defined as follows
\begin{align}
\label{eq:Sob} \mathcal{S}(A,\mathbf s,L)=\left\{f\in \lk^{1}(\R^{d}),\ \sum_{k=1}^{d}\int_{\R^{d}}|\mathcal F f(Au)|^{2}(1+|u_{k}|)^{2s_{k}}{\rm d} u \leq L \right\}.
\end{align}
For $f\in \mathcal{S}(A,\mathbf s,L), $ the bias term can be controlled by $$
\| f-f_{A,m}\|^{2}\leq  L|\mathrm{det}(A)| ( m_{1}^{-2s_{1}}+\cdots+ m_{d}^{-2s_{d}} ) \, .$$ 
Minimizing in $(m_{1},\ldots,m_{d})$ the  bound of Theorem \ref{thm:AD} we find for  
\begin{equation} \label{def:s}
\frac1{\overline s}=\sum_{k=1}^{d}\frac{1}{s_{k}}
\end{equation}
that the optimal cutoff is such that $m_{k}^{*}\asymp \left(\frac n{\log n}\right)^{\frac{2\overline s}{2s_{k}(2\overline {\mathbf s}+1)}}$ leading to a rate  $n^{-\frac{2\overline{\mathbf  s}}{2\overline{\mathbf  s}+1}}(\log n)^{\frac{2\overline{\mathbf  s}}{2\overline{\mathbf  s}+1}}.$
In case $A = Id$, this rate  is optimal, up to a power of $\log n$ (see Hasminskii and  Ibragimov (1990) \cite{10.1214/aos/1176347736}). Other regularity classes such as super smooth classes of densities can also be considered to control the bias term and different minimax optimal rates of convergence emerge (see Comte and Lacour (2013) \cite{comte2013anisotropic}).

 Note that $ \mathcal F f(Au)$ is the characteristic function of the random variable $Y = \, ^{t}AX$, if $X$ has Lebesgue density $f$. Introducing $A$  allows to define a regularity class dependent free of the specific frame in which the observations $X_{j}$ are displayed, since in Theorem \ref{thm:AD} we take the infimum over the set $\mathcal A$. Example \ref{ex:Ref} below shows that it is  a
relevant transformation that can improve the rates, since $A$ and  $\mathbf s$ are connected, namely $\mathbf s=\mathbf s(A)$. Taking the infimum over the set $\mathcal A$ permits to attain the rate induced by the best orientation of the axes, that is such that  $\overline{\mathbf s}=\overline{\mathbf s}(A)$ is maximal (see \eqref{def:s}).

The introduction of a linear transformation of the data is also evoked in Corollary 1 of Varet {\it et al.} (2019) \cite{varet2019numerical}: it is underlined that considering such transformations can improve the rates.
 However,  the matrix transformations considered therein are orthonormal and their procedure numerically investigates all possible transformations matrix whereas thresholding appears to intrinsically adapt to the optimal transformation.  
 
\begin{example}\label{ex:Ref}
Let $d=2$, $b>1$ and $0<a<b(1-b)$. Define $X=(bX_{1}, aX_{1}+bX_{2})$ where  $X_{1}$ and $X_{2}$ are independent, $X_{1}\sim\Gamma(\alpha+\frac12,1)$ and $X_{2}\sim\Gamma(\beta+\frac12,1)$, where $0<\beta<\alpha$. Denote by $f$ the density of $X$.  It is straightforward to check that 
$$
|\mathcal Ff(u)|^{2}=\frac1{(1+|bu_{1}+au_{2}|^{2})^{\alpha+\frac12}}\frac1{(1+|bu_{2}|^{2})^{\beta+\frac12}}. 
$$ 
Denote the associated bias  term, for $m_{1}\ge1$ and $m_{2}\ge1$, by
$B_{f}(m_1,m_2) = \iint_{[-m, m]^c} |\mathcal Ff(u)|^{2} {\rm d}u $. Consider the matrix 
$$A_{a,b}=\frac{1}{b^{2}}\begin{pmatrix}
b & -a\\ 0 & b
\end{pmatrix}\in \mathcal{A} \quad \text{for which} \quad 
|\mathcal Ff(A_{a,b}u)|^{2}=\frac1{(1+|u_{1}|^{2})^{\alpha+\frac12}}\frac1{(1+|u_{2}|^{2})^{\beta+\frac12}}$$ 
and the  associated bias term defined by 
$B_{f_{A_{a,b}}}(m_1,m_2) = \iint_{[-m, m]^c} |\mathcal Ff(A_{a,b}u)|^{2} {\rm d}u$.
We prove in Section \ref{sec:prf} the following result  for $0<\beta<\alpha$ and $m_{1}\ge1\vee C^{1}_{\alpha,\beta,b},\ m_{2}\ge1\vee C^{2}_{\alpha,\beta,b} $\begin{align}\label{eq:Bf2}
B_{f}(m_1,m_2)\ge\frac{c_{\alpha}}{2b\beta(1+b)^{2\beta}}\left(\frac{1}{m_2^{2\beta}}+ \frac{1}{2m_1^{2\beta}}\right)\ge  \frac{c_{\beta}}{2\alpha}\frac1{m_{1}^{2\alpha}}+ \frac{c_{\alpha}}{2\beta}\frac1{m_{2}^{2\beta}} \ge B_{f_{A_{a,b}}}(m_1,m_2),
\end{align} for some positive constants $C^{j}_{\alpha,\beta,b}$ and where  $c_{\gamma}=\int_{\R}{(1+z^{2})^{-(\gamma+\frac12)}} {\rm d}z $, $\gamma=\alpha,\beta$. 
Hence, if $0<\beta<\alpha$, the rate obtained by minimizing in $(m_1,m_2)$ the quantity
$$
18
\|f-f_{A,m}\|^{2}
+\frac{\left (10+2(1+(\kappa+2) \sqrt{\log n})^{2}\right )|\mathrm{det}(A)|}{\pi^d n}  m_{1} m_{2} 
$$
will be strictly better for $A=A_{a,b}$ than for $A=Id$ : for $A=A_{a,b}$, we obtain a rate of order $n^{-\frac{2\overline{\mathbf  s}}{2\overline{\mathbf  s}+1}}(\log n)^{\frac{2\overline{\mathbf  s}}{2\overline{\mathbf  s}+1}}$ with $\overline{\mathbf  s}=\alpha \beta/(\alpha +\beta)$, while for $A=Id$, we obtain a rate of order $n^{-\frac{\beta}{\beta+1}}(\log n)^{\frac{\beta}{\beta+1}}$.
\end{example}

\subsection{Empirical procedure to select $\kappa$\label{sec:kappa}}
Most adaptive methods require the calibration of a constant, here this constant is $\kappa$.
This is usually done by preliminary simulation experiments. Calibration strategies (dimension jump and slope heuristics) exist  for penalized methods, see Baudry \textit{et al.} (2012) \cite{baudry2012slope} and Lerasle (2012) \cite{lerasle2012optimal} for the\-o\-re\-ti\-cal justifications. 

Here, we propose a strategy to select $\kappa$ adapted to our context and we implement it in the numerical Sections \ref{sec:simiid} and \ref{sec:simdep}. This procedure is the result of two observations. Firstly, the set $\{|\hat{\phi}_{X,n}|\geq {\kappa}_n n^{-1/2}\}$ (recall that $\kappa_{n}= (1 +\kappa \sqrt{\log n}  )$) appearing in \eqref{eq:estfbis} is the excursion set of the process $Z = |\hat{\phi}_{X,n}|$ above the level ${\kappa}_n n^{-1/2}$. For these sets there exists an important literature on their geometry which provide an  information that we exploit here (see Adler and Taylor (2009) \cite{adler2009random}). One of these geometrical measure is the Euler characteristic;  in dimension $d=1$ it is a count of the number of connected components and if $d=2$ it is a count of the number of connected components minus the number of holes. 
Secondly, the constant $\kappa$ determines how much information from $\hat{\phi}_{X,n}$ is dropped. If $\kappa$ is selected too large relevant information on $f$ is lost and if $\kappa$ is selected too small  artefact noise may jeopardize its estimation. 

 We expect that the  function $\kappa\mapsto \chi(\{|\hat{\phi}_{X,n}|\geq {\kappa}_n n^{-1/2}\})$, where $\chi$ denotes the Euler characteristic, will stabilize once all uninformative areas are thresholded. Then, the first $\kappa$ where $\chi$ is constant will perform a good compromise between keeping enough information and removing most of the noise. 
This motivates the following adaptive procedure to select $\kappa$. Set $\delta >0$  and define $\chi_{x}:=\chi(\{|\hat{\phi}_{X,n}|\geq (1+x\delta\sqrt{\log n}) n^{-1/2}\})$ \begin{align}
\label{eq:hatkappa} \hat \kappa_{n}:= \inf\left\{\kappa \in \{k\delta, k\in \N\}, \chi_{\kappa} = \chi_{\kappa-1}\right\}.
\end{align} 
Note that this quantity does not depend on any additional calibration constant, except for the mesh-grid $\delta$  that should be taken  small. 

In the case $d=2$  we can visualize the set  $\{|\hat{\phi}_{X,n}|\geq {\kappa}_n n^{-1/2}\}$ as a black and white image (see Figure \ref{Fig:EC}). On this example $X$ is a Gamma-Beta random variable (see [GB] example in the next Section), we observe that if $\kappa$ is two small there are many uninformative areas that are taken into account in the estimator (left image) and that the Euler characteristic indeed gets constant for large enough $\kappa$ (right image).

\begin{figure}[H]\begin{center}
\begin{tabular}{ccc}
\includegraphics[scale = 0.47]{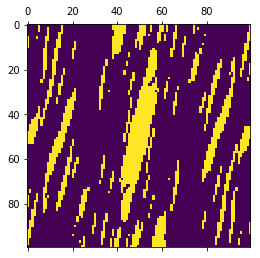}
&\includegraphics[scale = 0.47]{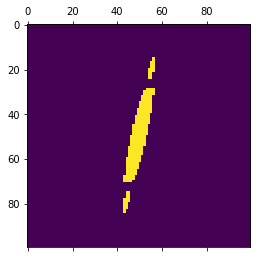}&\includegraphics[scale = 0.47]{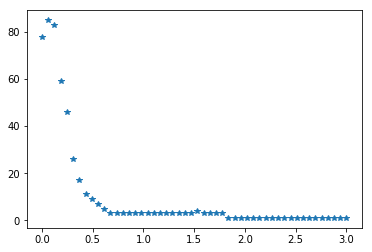}
\end{tabular}
\caption{\label{Fig:EC}\footnotesize{Representation of the set $\{|\hat{\phi}_{X,n}|\geq {\kappa}_n n^{-1/2}\}$ on a grid of $\R^{2}$ for the GB example 
for $\Sigma^{2}$
 defined in Section \ref{sec:simiid},
$n=10^{4}$, $\kappa = 0.1$ (left) and $\kappa = 1$ (center). Yellow pixels corresponds to areas where $\{|\hat{\phi}_{X,n}|\geq {\kappa}_n n^{-1/2}\}$. The last image (right) represents $\kappa\mapsto \chi(\{|\hat{\phi}_{X,n}|\geq {\kappa}_n n^{-1/2}\})$.}}
\end{center}
\end{figure}

\subsection{Numerical study in the independent case\label{sec:simiid}}

\paragraph{Examples considered}

We illustrate the performances of the estimator \eqref{eq:estfbis} with the adaptive choice of $\kappa$ described in \eqref{eq:hatkappa} in dimension $d=2$.
We consider three examples:\\
\indent [N] Gaussian: $X\sim \mathcal{N}\left(\begin{pmatrix}
0\\0
\end{pmatrix}, \Sigma^{2}\right)$, with $\Sigma^{2}=\begin{pmatrix}
1 & 0.5\\ 0.5 &3
\end{pmatrix}$,\\
\indent [Mix-NN] Mixture of Gaussian: $X\sim 0.4\,\mathcal{N}\left(\begin{pmatrix}
-2\\-2
\end{pmatrix},\begin{pmatrix}
1 & 0.2\\ 0.2 &3
\end{pmatrix}\right)+0.6\,\mathcal{N}\left(\begin{pmatrix}
2\\2
\end{pmatrix}, \begin{pmatrix}
1 & 0.3\\ 0.3 &1
\end{pmatrix}\right)$,\\
\indent [GB] Gamma-Beta : $X\sim W Y$, where $W= \begin{pmatrix}
1 & 0.1\\ 0.2 &1
\end{pmatrix}$ and $Y=\begin{pmatrix}
Y_{1}\\ Y_{2}
\end{pmatrix}$ where $Y_{1}\sim \Gamma (5,1)$ is independent of $Y_{2}\sim\mathcal{B}\mbox{eta}(2,2)$.\\
To evaluate the performances of the procedure we compute normalized $\lk^{2}$-risks   to allow the numerical comparaison of the different examples for which $\| f\|^{2}$ may vary. Namely, we evaluate empirically
\[\frac{\E[\|\hat{f}_n-f\|^{2}]}{\|f\|^{2}}\]
 and  the associated deviations, from $N=100$ independent datasets with different values of sample size $n=10^{3},\ 10^{4}$ and $10^{5}$. 

\paragraph{Results and comments}
They are displayed in Table \ref{Table:Risk} and confirm the theoretical results. As expected the estimated risks decrease as $n$ increases. Since we compute the normalized $\lk^{2}$-risks we can compare them for the different examples. As anticipated  the GB case is the most difficult to estimate and has the largest risk, indeed the Beta distribution is compactly supported which makes it difficult for a Fourier estimator to recover. Interestingly, the choice of $\kappa$ by \eqref{eq:hatkappa} leads in each example to a constant close to 1, except for the Gaussian case where it is a little  smaller. Its dependency in $n$ can also be observed: it has a tendency to decrease with $n$.

\begin{figure}[H]\begin{center}
\begin{tabular}{ccc}
\hspace{-0.8cm}\includegraphics[scale = 0.44]{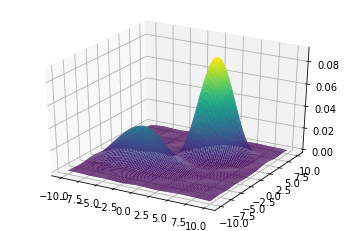}
&\includegraphics[scale = 0.44]{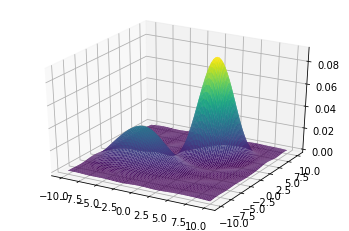}&\includegraphics[scale = 0.44]{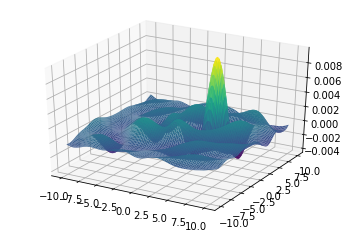}
\end{tabular}
\caption{\label{Fig:NN}\footnotesize{Representation for
$n=10^{4}$ of the estimated Mix-NN (left), the true density (center) and the difference of both estimated and true density (right).}}
\end{center}
\end{figure}

\begin{table}[H]\begin{center}
\begin{tabular}{c|c|c|c||c|c|c||c|c|c}
Case& \multicolumn{3}{c||}{N} & \multicolumn{3}{c||}{Mix-NN} &  \multicolumn{3}{c}{GB}\\
\hline
$n$ & $10^{3}$  & $10^{4}$  & $10^{5}$  & $10^{3}$  & $10^{4}$  & $10^{5}$  & $10^{3}$  & $10^{4}$  & $10^{5}$ \\
\hline
{100$\times$Risk } & $1.04$ & $0.12$  & $1.32$ {\scriptsize{$10^{-2}$}}  &
$3.02$  &$0.39$ &$4.70$ {\scriptsize{$10^{-2}$}} &$5.96$ &$1.65$  &$0.27$ \\
{100$\times\sqrt{\V}$(Risk) } &  {\footnotesize{$(0.56$)}} & {\footnotesize{$(0.05)$}} & {\footnotesize{$(0.58\times 10^{-2})$}} &  {\footnotesize{$(0.59)$}} &
 {\footnotesize{$(0.07)$}}  & {\footnotesize{$(1.05\times 10^{-2})$}} &  {\footnotesize{$(1.47)$}} &  {\footnotesize{$(0.24)$}} &  {\footnotesize{$(0.08)$}}\\
 \hline
\multirow{2}{*}{$\hat \kappa_{n}$} & 0.88 & 0.81 & 0.79& 1.06& 1.01& 0.94& 1.09&1.01& 1.00 \\
& {\footnotesize{$(0.19)$}} & {\footnotesize{$(0.15)$}} & {\footnotesize{$(0.13)$}} & {\footnotesize{$(0.09)$}} & {\footnotesize{$(0.09)$}}& {\footnotesize{$(0.09)$}}& {\footnotesize{$(0.08)$}} & {\footnotesize{$(0.09)$}}& {\footnotesize{$(0.15)$}} 
\end{tabular}
\caption{\label{Table:Risk}\footnotesize{Evaluated empirical risks multiplied by 100 and adaptive $\hat\kappa_{n}$, the associated squared variances are given in parenthesis.}}\end{center}
\end{table}

     \section{Directional adaptive procedure: mixing sequences\label{sec:dep}}
   In this section, we consider   $X_1,...,X_n$  a stationary sequence of ${\mathbb R}^d$-valued dependent random variables with density $f$ with respect to the Lebesgue measure. As in the independent case, we estimate $f$ on ${\mathbb R}^d$ with the estimator \eqref{eq:estfbis} and evaluate its performances using the quadratic loss on ${\mathbb R}^d$. We consider the cases of $\alpha$-mixing and $\tau$-mixing sequences,  for which we give analogues of Theorem \ref{thm:AD}. 

   \subsection{$\alpha$-mixing sequences}

To measure the  dependence between the variables $X_i$, we consider the usual $\alpha$-mixing coefficients introduced by  Rosenblatt (1956)   \cite{rosenblatt1956central}. Let us recall the definition of these coefficients.

\begin{definition} \label{def alpha}
	Let $(\Omega,\mathcal{A},\mathbb{P})$ be a probability space,  and $(X_i)_{i\in\mathbb{Z}}$ be a sequence of ${\mathbb R}^d$-valued random variables. Let $\mathcal{F}_k = \sigma (X_i : i \leq k)$ and $\mathcal{G}_k = \sigma (X_i : i \geq k)$. 
	The strong mixing coefficients $(\alpha(n))_{n \geq 0}$  of $(X_i)_{i \in \mathbb{Z}}$ (Rosenblatt \cite{rosenblatt1956central}), are defined by 
	\[   \alpha(0) = \frac12 \text{\; and \;} \alpha(n) = \sup_{k \in \mathbb{Z}}  \;\alpha( \mathcal{F}_k,\mathcal{G}_{k+n} ) \ \ \text{ for any } n> 0 ,\]
	where, for two $\sigma$-algebra ${\mathcal F}, {\mathcal G}$, 
	$$
	\alpha({\mathcal F}, {\mathcal G})=2 \sup_{A \in {\mathcal F}, B \in {\mathcal G}} \left | {\mathbb P}(A\cap B)- {\mathbb P}(A) {\mathbb P}(B)\right | \, .
	$$
	Note that we use here the convention of the book by Rio \cite{rio2017asymptotic}, so that the definition of $\alpha({\mathcal F}, {\mathcal G})$ differs from that of Rosenblatt \cite{rosenblatt1956central} from a factor 2.
\end{definition}

For $\alpha$-mixing sequences, we  prove the following oracle inequality :

\begin{theorem}\label{thm:ADalpha}
Let $ \kappa > 0$. Assume that  $\sum_{k=1}^{\infty} \alpha(k) < \infty $, 
and let $C_{\alpha}= 1 + 4 \sum_{k=1}^{\infty} \alpha(k)$.  
There exist constants $K_i, i \in \{ 1, ... , 7\}$, depending on $C_\alpha $ 
such that 
\begin{multline*}
\E[  \| \hat{f}_n- f\|^2 ]  \\
\leq  \underset{m\in[0,n]^{d}, A \in {\mathcal A}}{\inf}\left( K_1
\|f-f_{A, m}\|^{2} 
+\frac{K_2+\left (2(1+(\kappa+\sqrt {K_3})\sqrt{\log n})^{2} \right ) |\mathrm{det}(A)|}{\pi^d n}m_{1}\cdots m_{d}  \right)
\\  +\frac{2^7}{\pi^d} n^{d-\frac{\kappa^2}{K_3}} + f_\alpha(n,d,\kappa) \, ,
\end{multline*}
where the residual term $f_\alpha(n,d,\kappa)$ is given by 
$$
f_\alpha(n,d,\kappa)=\frac{K_4n^{(2d+1)/2}}{ \sqrt{\log n}}
 \alpha\left(\left[\frac{\sqrt{n} K_5 }{\sqrt{\log n}}\right ]\right)+\frac{K_6n^{(2d+1)/2}}{\kappa \sqrt{\log n}}
 \alpha\left(\left[\frac{\sqrt{n} K_7}{ \kappa \sqrt{\log n}}\right ]\right) \, .
$$
\end{theorem}

Explicit upper bounds for the constants $K_i, i \in \{ 1, ... , 7\}$, can be computed from the proof of Theorem \ref{thm:ADalpha}. 

\begin{remark}\label{rem:Dnbis} The term $n^{(2d+1)/2}$ in the expression of $f_\alpha(n,d,\kappa)$ implies that  the polynomial rate of mixing must depend on $d$ if this residual term is to be of order $O(1/n)$. More precisely, we need a rate of order $\alpha(n)=O(n^{-a})$, for $a>3+2d$. 
As in Remark \ref{rem:Dn}, if one does not look for the directional  adaptivity given by the infimum over the class ${\mathcal A}$, then the term  $n^{(2d+1)/2}$ can be considerably weakened by considering the estimator $\text{\v{f}}$  of Remark \ref{rem:Dn}. Doing so, the term $n^{(2d+1)/2}$ in the expression of $f_\alpha(n,d,\kappa)$ is replaced by $C_d n (\log n)^{d-1}$, and the constraint on the rate of mixing becomes $a>5$ (not depending on $d$). 
\end{remark}

According to Remark \ref{rem:Dnbis}, if $\alpha(n)=O(n^{-a})$ for $a>3+2d$ and for $\kappa$ large enough ($\kappa>\sqrt{K_3(d+1)}$),  the upper bound of Theorem \ref{thm:ADalpha} is the same, except for the constants, as in the independent case (Theorem \ref{thm:AD}). Then,  the resulting rates of convergence on
the Sobolev class $\mathcal S(A,\mathbf s,L)$ defined by \eqref{eq:Sob} are minimax optimal up to a logarithmic loss. Note that the latter choice for $\kappa$ depends on the unknown constant $C_{\alpha}$ that involves the mixing coefficients. In practice, the choice of $\kappa$ selected by the adaptive procedure described in Section \ref{sec:kappa} leads to very good numerical results and does not rely on any knowledge on $C_{\alpha}$.

Note that the residual term $f_\alpha(n,d,\kappa)$ tends to zero as soon as $\alpha(n)=O(n^{-a})$ for $a>2d +1$, and is of order 
$$
f_\alpha(n,d,\kappa)=O\left (  \frac{ (\log n)^{(a-1)/2}}{n^{(a-(2d+1))/2}}\right ) \, .
$$
Hence, our adaptive estimator will still be consistent as soon as $a>2d +1$, with possibly a suboptimal rate of convergence. In the simulations, for $d=1$, we shall look at the boundary case $a=3$. 

\smallskip 

To conclude this section, let us take a look at some  articles that consider the problem of adaptive density estimation in a dependent context. We are going to treat separately the articles which deal with $\beta$-mixing sequences, from those which deal with $\alpha$-mixing sequences, because the case of $\beta$-mixing is both simpler from a technical point of view, and more restrictive in terms of examples. Note that in some articles (such as \cite{lerasle2009adaptive,lerasle2011optimal}, \cite{bertin2017pointwise}) other notions of dependencies are also considered.

\smallskip

As mentioned in the Introduction,  the  only article that we know of that deals with adaptive estimation in a mixing context and in ${\mathbb R}^d$, is that of Bertin {\it et al.} \cite{BKLP20}. Note that the context and method are quite different from ours (kernel estimation on a bounded domain), and that these authors do not consider the anisotropic case and assume that the $\beta$-mixing coefficients decrease at an exponential rate.

All the articles we are going to discuss now only consider the case $d=1$.
 In the $\beta$-mixing framework, 
Tribouley and Viennet (1998) \cite{tribouley1998lp} proposed a wavelet method  and 
Comte and Merlev\`ede (2002) \cite{comte2002adaptive}  a general model selection procedure (valid for a large class of models) to estimate the density on a compact support. Both adaptive  estimators reach the minimax rates of convergence over Besov classes under the condition $\beta(k)=O(k^{-a})$ for  $a>4$. 
Comte et al (2008) \cite{CDT2008} proposed a model selection procedure for projection estimators on the Shannon basis (i.e. the orthonormal basis generated by the sinus-cardinal function), to estimate the density on the whole real line. The results are valid under the condition $\beta(k)=O(k^{-a})$ for  $a>4$. In  the same  context, Lerasle  \cite{lerasle2009adaptive,lerasle2011optimal} gave a general model selection result (valid for a large class of models) under the condition $\beta(k)=O(k^{-a})$ for  $a>3$. Note that in  \cite{lerasle2011optimal}, the delicate question of the ``data-driven" penalty is discussed.
Asin and Johannes (2018) \cite{asin2017adaptive} proposed a model selection procedure to estimate the density, when the regularity of $f$  is given by the decrease of the coefficients of the decomposition of $f$ on a fixed orthonormal basis. Here again, the results are valid if $\beta(k)=O(k^{-a})$ for some $a>4$ (see the comments on condition (4.6), after their Corollary 4.7).  In the $\alpha$-mixing framework, Comte and Merlev\`ede (2002) \cite{comte2002adaptive} also study the case where $\alpha(k)$ decreases geometrically (see their Theorem 3.2). Their adaptive estimator reaches the minimax rates of convergence over Besov classes, up to a power of a logarithmic term. In the case where $\alpha(k)=O(k^{-a})$ (see their Proposition 3.2), their results hold provided that $a>6$, and under an additional regularity assumption on the joint densities $g_{X_k, X_\ell}$, $k\neq \ell$.

All  previously cited articles give results for the integrated quadratic risk (integrated $\lk^{p}$-risk are also considered in  \cite{tribouley1998lp}), like the one we are considering here. Let us also mention the article by Bertin and Klutchnikoff (2017) \cite{bertin2017pointwise} which proposes an automatic bandwidth selection for a kernel estimator, and for the pointwise risk. Their result applies to a large class of dependent sequences,  in particular to $\alpha$-mixing sequences for which the mixing coefficient decreases geometrically.

\smallskip

We can see that, whether for  $\beta$ or $\alpha$ mixing sequences, the constraint $a>3$ on the mixing speed is always required for the adaptive estimation. This corresponds  to the minimum constraint  we found for $d=1$. Yet there is no heuristic explanation for this, since when the regularity is known, non adaptive estimators attain the minimax i.i.d.  rate as soon as $\sum \alpha(k) < \infty$. 

\smallskip

The article closest to ours in the case $d=1$ is that of Comte and Merlev\`ede (2002) \cite{comte2002adaptive}, although the context and the method of estimation are  different. The mixing speed we get to approach the minimax rate is a bit better than theirs ($a>5$ instead of $a>6$), but we still have a loss in a power of $\log n$. Moreover, we do not need any condition on the joint densities $g_{X_k, X_\ell}$.

\subsection{Geometrically $\tau$-mixing sequences} 
To measure the  dependence between the variables $X_i$, we consider the  $\tau$-mixing coefficients introduced by  Dedecker an Prieur    \cite{dedecker2005new}. Let us recall the definition of these coefficients.

\begin{definition} \label{def tau}
Let $(\Omega,\mathcal{A},\mathbb{P})$ be a probability space. Let $|\cdot|_2$ be the euclidean norm on ${\mathbb R}^d$. On $({\mathbb R}^d)^k$, consider the $\ell_1$ norm $|z|_{k,1}=|z_1|_2+ \cdots + |z_k|_2$.
Let  $Z$ be an $({\mathbb R}^d)^k$-valued random variable such that $\E(|Z|_{k,1})< \infty$. Let $\Lambda_1 (({\mathbb R}^d)^k)$ be the space of 1-Lipschitz function from $(({\mathbb R}^d)^k, |\cdot|_{k,1})$ to ${\mathbb R}$. 
The $\tau$-mixing coefficient between a $\sigma$-algebra ${\mathcal M}$  and the random variable $Z$ is defined by
\[
\tau ( {\mathcal M}, Z) = \left  \Vert  \sup \Big \{ \Big |  \int_{({\mathbb R}^d)^k} f(x) {\mathbb P}_{Z | { \mathcal M} } (dx) -  
\int_{({\mathbb R}^d)^k} f(x) {\mathbb P}_{Z} (dx) \Big | , f \in \Lambda_1 (({\mathbb R}^d)^k) \Big \}  \right \Vert_1 \, .
\]
Let now $(X_i)_{i\in\mathbb{Z}}$ be a sequence of ${\mathbb R}^d$-valued random variables, and let $\mathcal{F}_k = \sigma (X_i : i \leq k)$. Define then 
\begin{equation*} \label{deftaureverse}
  \tau (n) =\sup_{\ell \geq 1}  \frac 1 \ell \sup_{k+n\leq i_1 < \cdots < i_\ell} \tau( \sigma (X_i : i \leq k), (X_{i_1}, \ldots, X_{i_\ell})) \, .
\end{equation*}
\end{definition}

As stated in \cite{dedecker2005new}, the $\tau$-mixing  coefficient is a coupling coefficient for the Kantorovich distance. This property has two consequences : one can use the coupling properties to demonstrate very precise deviation inequalities for partial sums, and one can give a bound on  these coefficients for many classes of non-mixing processes in the sense of Rosenblatt (functions of i.i.d. sequences, non-irreducible Markov chains, dynamical systems... see again \cite{dedecker2005new}). As a recent exemple, let us mention that this coefficient can be computed for dynamical systems that can be modelled by Young Towers; for instance, if the return time to the base of the tower has an exponential moment, then the $\tau$-mixing coefficient decreases at an exponential rate (see \cite{CDM23}).

For $\tau$-mixing sequences whose coefficients decrease at an exponential rate, we consider the estimator $\hat{f}_n$ defined in \eqref{eq:estfbis} with a slight modification in the power of the logarithm in the threshold, namely we take $\kappa_n= 1+\kappa \log n$.

\begin{theorem}\label{thm:ADtau}
Assume that $\tau(n) \leq K a^n$ for some $K\geq 1$ and $a \in (0,1)$. There exist constants $K_1, K_2, K_3$ depending on $(K,a, d)$ and $K_4, b \in (a,1)$ depending on $(K,a, d, \kappa)$ such that (for $n\geq 2$)
\begin{multline*}
\E[  \| \hat{f}_n- f\|^2 ] \\
 \leq   \underset{m\in[0,n]^{d}, A \in {\mathcal A}}{\inf}\left( K_1 \log n
\|f-f_{A,m}\|^{2}
+\frac{\left (K_2 \log n +2(1+(\kappa+\sqrt{K_3})\log n)^{2}\right )|\mathrm{det}(A)|}{\pi^d n}m_{1}\cdots m_{d}  \right)
\\  +\frac{2^7}{\pi^d} n^{d-\frac{\kappa^2}{K_3}} + K_4 b^n \, .\end{multline*}
\end{theorem} 
Explicit upper bounds for the constants $K_1,K_2, K_3, K_4$ and $b$ can be computed from the proof of Theorem \ref{thm:ADtau}. 
We infer from Theorem \ref{thm:ADtau} that  the resulting rates of convergence of $\hat f_{n}$ on
the Sobolev class $\mathcal S(A,\mathbf s,L)$ defined by \eqref{eq:Sob} are minimax optimal up to a $(\log n)^2$ term.

\begin{remark}
In case of subexponential decay, that is $\tau(n) \leq K a^{n^\gamma}$ for some $K\geq 1$, $a \in (0,1)$ and $\gamma \in (0, 1)$, one can take $\kappa_n= 1+\kappa (\log n)^{(1+\gamma)/2\gamma}$. This lead to the upper bound 
\begin{multline*}
\E[  \| \hat{f}_n- f\|^2 ]  \leq  \underset{m\in[0,n]^{d}, A \in {\mathcal A}}{\inf}\Bigg( K_1 (\log n)^{\frac 1\gamma}
\|f-f_{A,m}\|^{2} \\
\hspace{4cm}+\frac{\left (K_2 (\log n)^{\frac 1\gamma} +2(1+(\kappa+\sqrt{K_3})(\log n)^{(1+\gamma)/2\gamma})^{2}\right )|\mathrm{det}(A)|}{\pi^d n}m_{1}\cdots m_{d}  \Bigg)
\\  +\frac{2^7}{\pi^d} n^{d-\frac{\kappa^2}{K_3}} + K_4 b^{n^\gamma} 
\end{multline*}
for some  $b \in (a,1)$. We refer to \cite{CDM23} for  examples of sequences for which $\tau(n)$ decreases at a subexponential rate.
\end{remark}

To conclude this section, note that, in the case $d=1$,  Lerasle  \cite{lerasle2009adaptive} gave a  model selection result for a wavelet type estimator under the condition $\tau(k)=O(k^{-a})$ for  $a>6$ (see also his Remark 4.4 in case $a>4$).

 \subsection{Numerical study in the dependent case\label{sec:simdep}}
\paragraph{$\alpha$-mixing sequences}

We illustrate our adaptive procedure in dimension 1.
We expect that our empirical  procedure to select  $\kappa$ will adapt to the unknown value of $C_{\alpha}$.

To simulate $\alpha$-mixing sequences, we use the Markov chain introduced by Doukhan et al. (1994) \cite{doukhan1994functional}. The interest of this chain is  first that it is easy to simulate (as many Markov chains for which one can exhibit explicitly the iterated random system), and secondly we can compute exactly its rate of mixing from the parameters of the transition kernel. Since the chain is irreducible, positively recurrent ans aperiodic, it is also $\beta$-mixing in the sense of  Volkonskii and Rozanov (1959) \cite{volkonskii1959some}. Recall that, for such Markov chains, $\beta$-mixing is equivalent to $\alpha$-mixing. 

Let us now describe this Markov chain in more details. 
Let $a>1$ (as before $a$ will calibrate the rate of mixing), and let  $\mu$ and $\nu$ be the two probability measures on $[0,1]$ with respective densities 
$$f_\mu(x)=a x^{a-1}{\bf 1}_{0 \leq x \leq 1} \quad \text{and} \quad  f_\nu(x)=(a+1) x^{a}{\bf 1}_{0 \leq x \leq 1}\, .$$ Let $F_{\nu}$ be the cumulative distribution function of $\nu$, and let $Y_1$ be a random variable with law $\mu$. Let $(\varepsilon_j)_{j \geq 1}=((U_j, V_j))_{j \geq 2}$ be a sequence of i.i.d. random variables with uniform law over $[0,1]^2$ and independent of $Y_1$. For $j \geq 1$ define then
$$
Y_{j+1}=Y_{j} {\bf 1}_{U_{j+1} \geq Y_{j}}+F_\nu^{-1}(V_{j+1})  {\bf 1}_{U_{j+1} <Y_{j}} \, .
$$
It is proved in \cite{doukhan1994functional} that this chain is strictly stationary, with invariant distribution $\mu$ and that the $\beta$-mixing coefficients of this chain are exactly of order $n^{-a}$. Note also that $Y_j^a$ is uniformly distributed over $[0,1]$. To derive from this chain a sequence with invertible cumulative distribution function $F$, we take 
$X_j=F^{-1}(Y_j^a)$.
The sequence $(X_j)_{j\geq 1}$ is also a stationary Markov chain (as an invertible function of a stationary Markov chain), and its $\beta$-mixing coefficients are such that: there exist  two positive constants  $B>A>0$ such that, for any $n \geq 1$,  $$\frac{A}{n^{a}} \leq \beta(n) \leq \frac{B}{n^{a}}$$
(and the same is true for the coefficients $\alpha(n)$ for different constants $A,B$). \\

 As in Section \ref{sec:simiid} we evaluate the  normalized $\lk^{2}$-risks from $500$ independent datasets, for the sample sizes $n=500,\ 2000$ and 5000 and different mixing coefficients $a\in\{3,6,10\}$. Note that $a>3$ is the minimal value for which we proved that our estimator is consistent (without being necessarily optimal). It is then interesting to see if, numerically, our procedure still works at the boundary $a=3$. 
If $a>5$, Theorem \ref{thm:ADalpha} shows that our estimator reaches  the minimax rate of the i.i.d. case, up to a logarithmic term.

\begin{figure}
	\centering
	\begin{tabular}{cccc}		& $a=3$ & $a=6$  &  $a=10$ \\ 
		\begin{turn}{90}$n=500$\end{turn}&
		\includegraphics[scale=0.3]{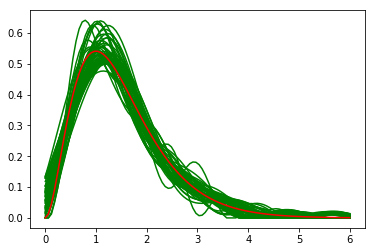}&\includegraphics[scale=0.3]{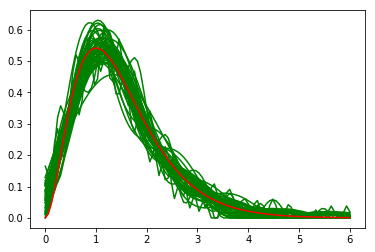}&  \includegraphics[scale=0.3]{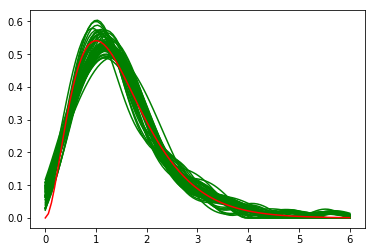} \\ 
	\begin{turn}{90}$n=2000$\end{turn}&	\includegraphics[scale=0.3]{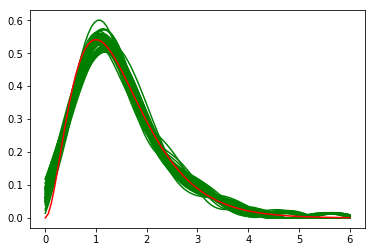}&\includegraphics[scale=0.3]{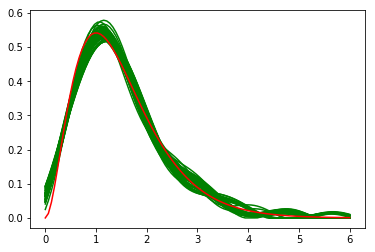}&  \includegraphics[scale=0.3]{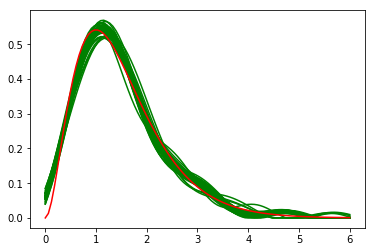} \\ 
	\begin{turn}{90}$n=5000$\end{turn}&	\includegraphics[scale=0.3]{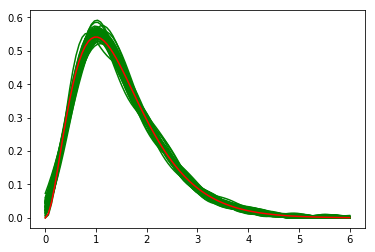}&\includegraphics[scale=0.3]{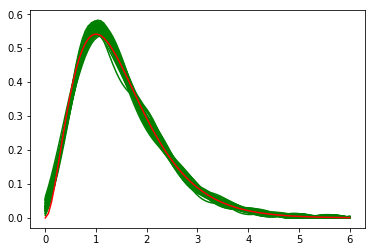}&  \includegraphics[scale=0.3]{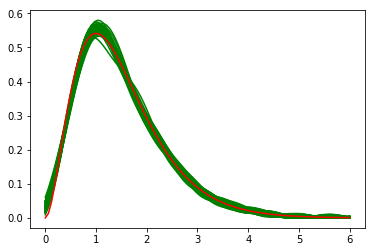} \\ 
	\end{tabular} 
	\caption{\footnotesize{Estimation of $f \sim \Gamma(3,2)$ (red) and the adaptive estimator (green) from $ 50$ Monte Carlo iterations  for different values of $a$ and $n$.
	\label{fig:gamma}}}
\end{figure}

\begin{table}\begin{center}
\begin{tabular}{c|c|c|c||c|c|c||c|c|c}
$n$& \multicolumn{3}{c||}{500} & \multicolumn{3}{c||}{2000} &  \multicolumn{3}{c}{5000}\\
\hline
$a$ &  3 & 6 & 10 &3 & 6 & 10& 3 & 6 & 10  \\
\hline
{100$\times$Risk } & $1.66$ & $1.37$  & 1.20  & 0.57 & 0.49 & 0.42 & 0.28 & 0.22 & 0.19\\
{100$\times\sqrt{\V}$(Risk) } &  {\footnotesize{$(0.88$)}} & {\footnotesize{$(0.65)$}} & {\footnotesize{$(0.59)$}} &  {\footnotesize{$(0.28)$}} &
 {\footnotesize{$(0.23)$}}  & {\footnotesize{$(0.20)$}} &  {\footnotesize{$(0.13)$}} &  {\footnotesize{$(0.10)$}} &  {\footnotesize{$(0.08)$}}\\
 \hline
\multirow{2}{*}{$\hat \kappa_{n}$} & 0.74 & 0.62 & 0.55& 0.72& 0.61& 0.57& 0.73&0.61& 0.46 \\
& {\footnotesize{$(0.26)$}} & {\footnotesize{$(0.22)$}} & {\footnotesize{$(0.20)$}} & {\footnotesize{$(0.24)$}} & {\footnotesize{$(0.22)$}}& {\footnotesize{$(0.23)$}}& {\footnotesize{$(0.27)$}} & {\footnotesize{$(0.23)$}}& {\footnotesize{$(0.19)$}} 
\end{tabular}
\caption{\footnotesize{Computations via $500$ Monte Carlo iterations  of the $\mathbb{L}^2$ risks for the Gamma distribution  $\Gamma(3,2)$. \label{tab1}}}\end{center}
\end{table}

\begin{figure}
	\centering
	\begin{tabular}{cccc}		& $a=3$ & $a=6$  &  $a=10$ \\ 
		\begin{turn}{90}$n=500$\end{turn}&
		\includegraphics[scale=0.3]{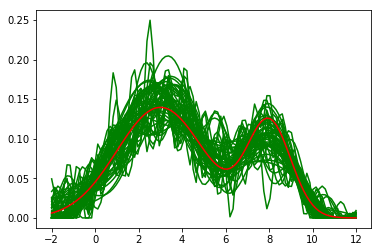}&\includegraphics[scale=0.3]{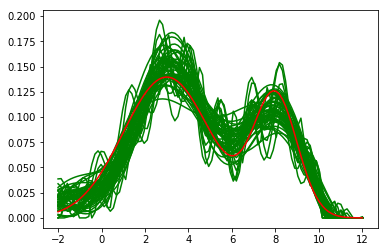}&  \includegraphics[scale=0.3]{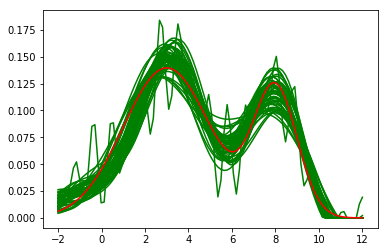} \\ 
	\begin{turn}{90}$n=2000$\end{turn}&	\includegraphics[scale=0.3]{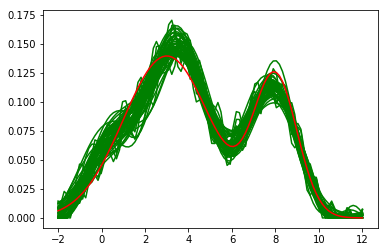}&\includegraphics[scale=0.3]{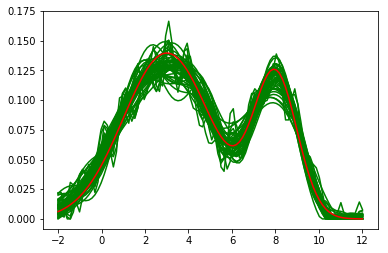}&  \includegraphics[scale=0.3]{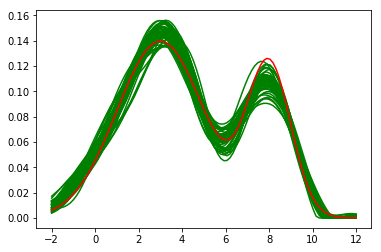} \\ 
	\begin{turn}{90}$n=5000$\end{turn}&	\includegraphics[scale=0.3]{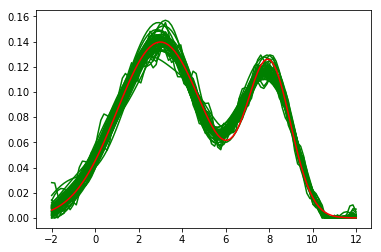}&\includegraphics[scale=0.3]{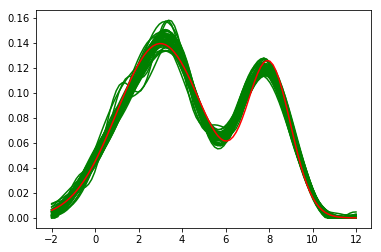}&  \includegraphics[scale=0.3]{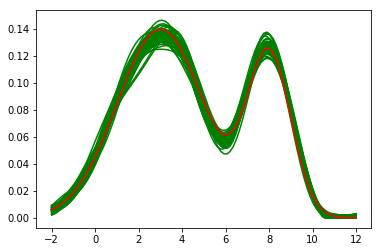} \\ 
	\end{tabular} 
	\caption{\footnotesize{Estimation of $f \sim 0.7\mathcal{N}(3,2) +0.3\mathcal{N}(8,1)$ (red) and the adaptive estimator (green) from $ 50$ Monte Carlo iterations  for different values of $a$ and $n$.
	\label{fig:NN}}}
\end{figure}

\begin{table}\begin{center}
\begin{tabular}{c|c|c|c||c|c|c||c|c|c}
$n$& \multicolumn{3}{c||}{500} & \multicolumn{3}{c||}{2000} &  \multicolumn{3}{c}{5000}\\
\hline
$a$ &  3 & 6 & 10 &3 & 6 & 10 &3 & 6 & 10  \\
\hline
{100$\times$Risk } & $3.32$ & $2.62$  & 2.44  & 0.91 & 0.72 & 0.65 & 0.51 & 0.38 & 0.34\\
{100$\times\sqrt{\V}$(Risk) } &  {\footnotesize{$(1.68$)}} & {\footnotesize{$(0.69)$}} & {\footnotesize{$(0.91)$}} &  {\footnotesize{$(0.43)$}} &
 {\footnotesize{$(0.34)$}}  & {\footnotesize{$(0.24)$}} &  {\footnotesize{$(0.23)$}} &  {\footnotesize{$(0.15)$}} &  {\footnotesize{$(0.13)$}}\\
 \hline
\multirow{2}{*}{$\hat \kappa_{n}$} & 0.76 & 0.69 & 0.68& 0.71& 0.62& 0.58& 0.76&0.67& 0.64 \\
& {\footnotesize{$(0.16)$}} & {\footnotesize{$(0.18)$}} & {\footnotesize{$(0.20)$}} & {\footnotesize{$(0.16)$}} & {\footnotesize{$(0.17)$}}& {\footnotesize{$(0.16)$}}& {\footnotesize{$(0.15)$}} & {\footnotesize{$(0.15)$}}& {\footnotesize{$(0.16)$}} 
\end{tabular}
\caption{\footnotesize{Computations by $500$ Monte Carlo iterations  of the $\mathbb{L}^2$ risks for the  distribution  $ 0.7\mathcal{N}(3,2) +0.3\mathcal{N}(8,1)$. }
	\label{tab2}}	\end{center}\end{table}

We consider two distributions  Gamma $\Gamma(3,2)$ and the mixture $0.7\mathcal{N}(3,2) +0.3\mathcal{N}(8,1)$.
On Figures \ref{fig:gamma} and \ref{fig:NN}, we observe  the behavior of $50$ adaptive estimators around the  true density distribution. We observe that even for values of $a$ smaller than 5 the estimator proposed is relevant and as expected we observe an improvement of the results when $a$ and/or $n$ increase. The risks empirically estimated in  Tables \ref{tab1} and \ref{tab2} confirm these observations. It is interesting to notice that for fixed values of $a$ the selected values for $\kappa$ are similar, even when $n$ increases, and that numerically we can  proceed  without knowing the constant $C_{\alpha}$.

\paragraph{$\tau$-mixing sequence} \label{sub3.3}
We consider a stationary Markov chain $(X_j)$  simulated  according to the following  auto-regressive mechanism:
$$\begin{cases}
X_0 &\sim\ \mathcal{U}\left[ 0,1 \right] \\
X_j&=\ \frac{1}{2}\left( X_{j-1}+\varepsilon_j\right) ,\quad j\ge 1,
\end{cases}$$
where $(\varepsilon_j)_{j\geq 1}$ are i.i.d., with common distribution $\mathcal{B}\left( \frac{1}{2} \right)$,  and  independent of $X_0$.
The $X_i$'s are uniformly distributed over $[0,1]$. 
To generate another stationary Markov chain with invariant cumulative distribution function $F$, we apply $F^{-{1}}$ to this sequence. 
It is well known that the Markov chain $(X_j)_{j \geq 0}$ is not $\alpha$-mixing 
in the sense of Rosenblatt \cite{rosenblatt1956central} (see for instance \cite{andrews1984non}). However, one can prove that other dependence coefficients, such as the $\tau$-dependence coefficients defined in Definition \ref{def tau}, converge to 0 at an exponential rate (an easy argument shows that it is still the case for the sequence $F^{-1}(X_i)$ for quantile functions $F^{-1}$ that are Lipschitz on any interval $I_\varepsilon=[\varepsilon, 1- \varepsilon]$, $\varepsilon >0$, provided the Lipschitz constant on $I_\varepsilon$ does not grows faster than $C\varepsilon^{-\gamma}$ for some $C>0, \gamma >0$).

\begin{figure}[H]
	\centering
	\begin{tabular}{ccc}		$n=500$ & $n=2000$  &  $n=5000$ \\ 
		
		\includegraphics[scale=0.3]{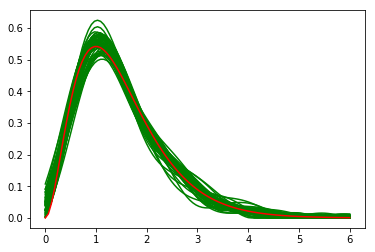}&\includegraphics[scale=0.3]{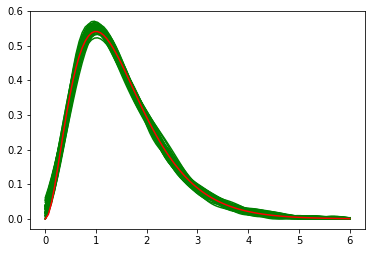}&  \includegraphics[scale=0.3]{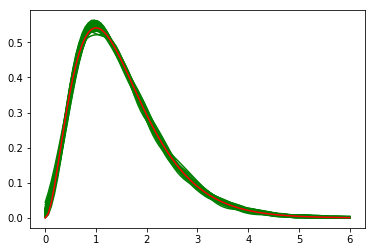} \\ 
	\end{tabular} 
	\caption{\footnotesize{Estimation of $f \sim \Gamma(3,2)$ (red) and the adaptive estimator (green) from $ 50$ Monte Carlo iterations  for different values of  $n$.
	\label{fig:ARgamma}}}
\end{figure}

\begin{table}[H]
\begin{center}	\begin{tabular}{c|c|c|c}
	$n$	& 	$500$ & $2000$ & $5000$ \\ 
		\hline 
		\multirow{2}{*}{Risk} &8.62 $\times 10^{-3}$ &  2.57 $\times 10^{-3}$&1.25 $\times 10^{-3}$ \\ 
		 &  \small{($5.40\times 10^{-3}$)}  &  \small{($1.39\times 10^{-3}$)} &  \small{(0.67$\times 10^{-3}$)}\\\hline
		\multirow{2}{*}{$\hat\kappa_{n}$}  & 0.12 &0.12& 0.13\\ 
		 &  \small{(0.06)}  &  \small{(0.05)} &  \small{(0.05)}\\
	\end{tabular} 
	\caption{\footnotesize{Computations via $500$ Monte Carlo iterations of the $\mathbb{L}^2$ risks for a Markov chain generated from the non mixing auto-regressive model with $f \sim \Gamma(3,2)$.   
	\label{tab5}}}\end{center}\end{table}

   \section{Proofs}\label{sec:prf}

\subsection{Proof of Theorem \ref{thm:AD}\label{sec:prfiid} }

\subsubsection{Preliminaries}

We start with two technical lemmas, an Hoeffding inequality for complex valued random variable and a lemma that enables to control  variance terms in the proof of Theorem \ref{thm:AD}.

\begin{lemma}\label{lem:Hoeffding}
Let $b>0$.  For all $u\in \R^{d},\ d\ge 1$, the following inequality holds
$$ \PP\left(|\hat \phi_{X,n}(u)-\phi_{X}(u)|\ge \frac{b \sqrt{\log n}}{\sqrt n}\right)\le 4 n^{-\frac{b^{2}}{4}}.$$
\end{lemma}
\begin{proof}[Proof of Lemma \ref{lem:Hoeffding}]
We apply the Hoeffding inequality to the centered variable
$$
\hat \phi_{X,n}(u)-\phi_{X}(u)= \frac{1}{n}\sum_{j=1}^{n}\left(\cos(\langle u,X_{j}\rangle )-\mbox{Re}(\phi_{X}(u))\right)+i\frac{1}{n}\sum_{j=1}^{n}\left(\sin(\langle u,X_{j}\rangle )-\mbox{Im}(\phi_{X}(u))\right).
$$ 
By standard arguments, we first note that, for two real-valued random variables $A,B$ and $x>0$,
\begin{align}
{\mathbb P}\left (  \sqrt{A^2+B^2}\geq x \right ) \leq {\mathbb P}(|A|\geq x/\sqrt 2) + {\mathbb P}(|B|\geq x/\sqrt 2) \, .
\label{eq:84}
\end{align}
Using \eqref{eq:84}, we get 
\begin{align*}
\PP\left(|\hat \phi_{X,n}(u)-\phi_{X}(u)|\ge \frac{b \sqrt{\log n}}{\sqrt n}\right)\le &\PP\left(\left|\frac{1}{n}\sum_{j=1}^{n}\left(\cos(\langle u,X_{j}\rangle )-\mbox{Re}(\phi_{X}(u))\right)\right|\ge \frac{b \sqrt{\log n}}{\sqrt{2n}}\right)\\ &+\PP\left(\left|\frac{1}{n}\sum_{j=1}^{n}\left(\sin(\langle u,X_{j}\rangle )-\mbox{Im}(\phi_{X}(u))\right)\right|\ge \frac{b\sqrt{\log n}}{\sqrt {2n}}\right)\\
\le& 4\exp\left(-\frac{b^{2}\log n}{4}\right).
\end{align*}
\end{proof}

\begin{lemma}\label{lem:var} Let $m=(m_{1},\ldots,m_{d})\in(0,n]^d$. For any $\kappa>0$ and any $A \in {\mathcal A}$, the following inequality holds 
\begin{align}
\label{eq:UBprf21}
\frac{1}{(2\pi)^d}\int_{A([-m,m])}\E\big[|\tilde \phi_{X,n}(u)-\phi_{X}(u)|^{2}\big]{\rm d}u\le \frac{\left (5+(1+(\kappa+2)\sqrt{\log n})^{2}\right ) |\mathrm{det}(A)|}{\pi^d n}m_{1}\cdots m_{d}.
\end{align}
\end{lemma}

\begin{proof}[Proof of Lemma \ref{lem:var}]
First note that
\begin{equation}\label{eq:var1}
\E\big[|\tilde \phi_{X,n}(u)-\phi_{X}(u)|^{2}\big]=\E\big[|\hat \phi_{X,n}(u)-\phi_{X}(u)|^{2}\mathbf{1}_{|\hat\phi_{X,n}(u)|\geq\frac{\kappa_{n}}{\sqrt{n}}}\big]+|\phi_{X}(u)|^{2}\PP\Big(|\hat\phi_{X,n}(u)|< \frac{\kappa_{n}}{\sqrt{n}}\Big).
\end{equation}
 The first term in the right hand side is bounded by $\frac{1}{n}$.
Recall that $\kappa_{n}=1+\kappa\sqrt{\log n}$, we decompose the second term on the set $\{u,|\phi_{X}(u)|< \frac{1+(\kappa+2)\sqrt{\log n}}{\sqrt{n}}\}$ and its complementary where we derive from the triangle inequality that $$\mathbf{1}_{|\hat\phi_{X,n}(u)|< \frac{1+\kappa\sqrt{\log n}}{\sqrt{n}}}\mathbf{1}_{|\phi_{X}(u)|> \frac{1+(\kappa+2)\sqrt{\log n}}{\sqrt{n}}}\le \mathbf{1}_{|\hat\phi_{X,n}(u)-\phi_{X}(u)|\geq \frac{2\sqrt{\log n}}{\sqrt{n}}}.$$ This leads to
\begin{align}\label{eq:var2}
\hspace{-0.5cm}|\phi_{X}(u)|^{2}\PP\Big(|\hat\phi_{X,n}(u)|< \frac{\kappa_{n}}{\sqrt{n}}\Big)&
\leq \frac{(1+(\kappa+2)\sqrt{\log n})^{2}}{n}+\PP\Big(|\hat\phi_{X,n}(u)-\phi_{X}(u)|\geq \frac{2\sqrt{\log n}}{\sqrt{n}}\Big) \nonumber \\
&\leq  \frac{(1+(\kappa+2)\sqrt{\log n})^{2}}{n}+\frac 4 n \leq  \frac{4+(1+(\kappa+2)\sqrt{\log n})^{2}}{n}, 
\end{align} where we used  Lemma \ref{lem:Hoeffding} with $b=2$. To conclude the proof, it suffices to note that 
\begin{equation}\label{eq:detA}
\int {\mathbf 1}_{A([-m,m])}(u) {\rm d}u = |\mathrm{det}(A)|m_{1}\cdots m_{d} \, .
\end{equation}
\end{proof}

\subsection{Proof of Theorem \ref{thm:AD} } Let $m=(m_{1},\ldots,m_{d})\in(0,n]^d$. We first note that  $$ \E[\|\hat{f}_n-f\|^{2}]\le 2\|f_{A, m}-f\|^{2}+2\E[\|\hat{f}_n-f_{A, m}\|^{2}]= \frac2{(2\pi)^{d}}\int_{A([-m,m])^{c}}| \phi_{X}(u)|^{2}{\rm d}u+2\E[\|\hat f_n-f_{A,m}\|^{2}].$$
The first term is a bias term and we treat the second variance term using the Parseval equality to get that 
\begin{multline}2\E[\|\hat f_n-f_{A,m}\|^{2}]=\frac2{(2\pi)^{d}}\int_{\R^{d}}\E\big[|\tilde \phi_{X}(u)\mathbf{1}_{[-n,n]^{d}}(u)-\phi_{X}(u)\mathbf{1}_{A([-m,m])}(u)|^{2}\big]{\rm d}u  \\
=\frac2{(2\pi)^{d}}\int_{A([-m,m])}\E\big[|\tilde \phi_{X,n}(u)-\phi_{X}(u)|^{2}\big]{\rm d}u+\frac2{(2\pi)^{d}}\int_{[-n,n]^{d}\setminus A([-m,m])}\E\big[|\tilde \phi_{X,n}(u)|^{2}\big]{\rm d}u.
\label{eq:step1bis}\end{multline} 
The first term is a variance term which is controlled with \eqref{eq:UBprf21} for $\kappa>0$. In the sequel we focus on the second term. To that end notice that
\begin{align*}
|\tilde\phi_{X,n}|^2   \leq 2 |{\phi}_{X}|^2
+ 2 | \tilde{\phi}_{X,n}-\phi_X|^2.
\end{align*}From which we get
\begin{align*}
\frac2{(2\pi)^{d}}\int_{[-n,n]^{d}\setminus A([-m,m])}\E\big[|\tilde \phi_{X,n}(u)|^{2}\big]{\rm d}u\le& \frac4{(2\pi)^{d}}\int_{[-n,n]^{d}\setminus A([-m,m])}| \phi_{X}(u)|^{2}{\rm d}u\\ &+\frac4{(2\pi)^{d}}\int_{[-n,n]^{d}\setminus A([-m,m])}\E\big[|\tilde \phi_{X,n}(u)-\phi_{X}(u)|^{2}\big]{\rm d}u\\
\le& \frac4{(2\pi)^{d}}\int_{A([-m,m])^{c}}| \phi_{X}(u)|^{2}{\rm d}u\\ &+\frac4{(2\pi)^{d}}\int_{[-n,n]^{d}\setminus A([-m,m])}\E\big[|\tilde \phi_{X,n}(u)-\phi_{X}(u)|^{2}\big]{\rm d}u.
\end{align*} The first term is a bias term, and we treat the last one by considering the set
 $\{u, |\phi_{X}(u)|>n^{-1/2}\}$ and its complementary. On the set $\{ u, |\phi_X(u)| > n^{-1/2} \}$, we derive that 
\begin{align}\label{eq:VtoB}
 \E[|\tilde{\phi}_{X,n}(u)- \phi_X(u)|^2]&\leq|\phi_{X}(u)|^{2}+ \E[|\hat{\phi}_{X,n}(u)- \phi_X(u)|^2] \leq|\phi_{X}(u)|^{2}+\frac1n\leq 2|\phi_{X}(u)|^{2}.
 \end{align}
Consequently, we recover a bias term since \begin{align*}
\frac4{(2\pi)^{d}} \int_{[-n,n]^{d}\setminus A([-m,m])}&\E\big[|\tilde \phi_{X,n}(u)-\phi_{X}(u)|^{2}\big]\mathbf{1}_{\{ |\phi_X(u)| > n^{-1/2} \} } {\rm d}u\le \frac8{(2\pi)^{d}}\int_{A([-m,m])^{c}}|\phi_{X}(u)|^{2}{\rm d}u.
 \end{align*}
 Next on the complementary, using that $|\tilde{\phi}_{X,n}(u)- \phi_X(u)|^2\leq 4$ and the definition of $\tilde \phi_{X,n}$, we derive   that 
 \begin{multline}
\frac4{(2\pi)^{d}}\E  \Big[ \int_{[-n,n]^{d}\setminus A([-m,m])} \limits\hspace{-0.5cm}  |\tilde{\phi}_{X,n}(u)- \phi_X(u)|^2 \mathbf{1}_{\{|\phi_X(u)| \leq  n^{-1/2}\}}{\rm d}u\Big] \\
\leq   \frac4{(2\pi)^{d}} \int_{[-n,n]^{d}\setminus A([-m,m])}\limits  \hspace{-0.5cm}  |\phi_X(u)|^2 {\rm d}u  + \frac{16}{(2\pi)^{d}}  \int_{[-n,n]^{d}\setminus A([-m,m])}\limits \hspace{-0.5cm}\PP(|\hat\phi_{X,n}(u)|\geq \kappa_{n}n^{-1/2})\mathbf{1}_{\{|\phi_{X}(u)|\leq n^{-1/2}\}} {\rm d}u \\
\leq   \frac4{(2\pi)^{d}}\int_{A([-m,m])^{c} }\limits    |\phi_X(u)|^2 {\rm d}u + \frac{16}{(2\pi)^{d}}  \int_{[-n,n]^{d}}\limits \PP( |\hat{\phi}_{X,n}(u) - \phi_X(u)| > \kappa (\log n/n)^{1/2}  )   {\rm d}u\\   \leq \frac4{(2\pi)^{d}} \int_{A([-m,m])^{c} }\limits   |\phi_X(u)|^2 {\rm d}u +  \frac{64}{\pi^{d}}  n^{ d - \kappa^2/4}. \label{eq:VtoBH}
 \end{multline}
  The last inequality is a direct consequence of  Lemma \ref{lem:Hoeffding} with $b=\kappa$. 
 Gathering all terms it follows  
 that  for all $m\in(0,n]^{d}$,
\begin{align*}
\E[  \| \hat{f}_n- f\|^2 ]  \leq  18
\|f-f_{A, m}\|^{2}
+\frac{\left (10+2(1+(\kappa+2)\sqrt{\log n})^{2}\right ) |\mathrm{det}(A)| }{\pi^d n}m_{1}\cdots m_{d} +  \frac{64}{\pi^{d}}n^{d - \kappa^2/4}.    
\end{align*}
Taking the infimum over $m$ completes the proof.

\subsection{Proof of Equation \eqref{eq:VolDn}}
Firstly, note that an inductive argument allows to show that the density of $V_{d}=U_{1}\ldots U_{d}$, $d\ge 1$, where  $(U_{j})_{1\le j\le d}$ are i.i.d. $\mathcal{U}([0,1])$ random variables,  is given by
$$ g_{d}(z)=\frac{(-1)^{d}(\log z)^{d-1}}{(d-1)!}\mathbf{1}_{[0,1]}(z).$$
Secondly, we compute the quantity in \eqref{eq:VolDn} using the latter result and a change  of variables
\begin{align*}
\int_{D_{n}}{\rm d}u &= \int_{[-n,n]^{d}}\mathbf{1}_{|u_{1}\ldots u_{d}|\le n}{\rm d}u_{1}\ldots{\rm d}u_{d} = 2^{d} \int_{[0,n]^{d}}\mathbf{1}_{|u_{1}\ldots u_{d}|\le n}{\rm d}u_{1}\ldots{\rm d}u_{d}\\
&= 2^{d}n^{d} \int_{[0,1]^{d}}\mathbf{1}_{|u_{1}\ldots u_{d}|\le n^{1-d}}{\rm d}u_{1}\ldots{\rm d}u_{d} = (2n)^{d}\PP(V_{d}\le n^{1-d})\\
&=(2n)^{d}\int_{0}^{n^{1-d}}\frac{(-1)^{d}(\log z)^{d-1}}{(d-1)!}\mathbf{1}_{[0,1]}(z){\rm d}z\\
&=\frac{(2n)^{d}}{(d-1)!}\int_{\log(n^{d-1})}^{\infty}z^{d-1} e^{-z}{\rm d}z = \frac{(2n)^{d}}{(d-1)!}\Gamma(d, \log(n^{d-1}))
\end{align*} where we made the change of variable $z\leftarrow -\log(z)$ and $\Gamma(s,x)$ is the incomplete Gamma function. Using the {asymptotic }expansion $ \Gamma(s,x)\underset{x\to\infty}{\sim } x^{s-1 }e^{-x}$ we get
\begin{align*}
\int_{D_{n}}{\rm d}u &\underset{n\to\infty}{\sim }  \frac{(2n)^{d}}{(d-1)!}(\log(n^{d-1}))^{d-1}e^{-\log(n^{d-1})}= \frac{2^{d}(d-1)^{d-1}}{(d-1)!}\log(n)^{d-1}n.
\end{align*} 

\subsection{Proof of Equation \eqref{eq:Bf2} in Example \ref{ex:Ref}}

Let $d=2$, $b>1$ and $0<a<b(1-b)$, define $X=(bX_{1}, aX_{1}+bX_{2})$ where  $X_{1}$ and $X_{2}$ are independent, $X_{1}\sim\Gamma(\alpha+\frac12,1)$ and $X_{2}\sim\Gamma(\beta+\frac12,1)$, where $0<\beta<\alpha$. Denote by $f$ the density of $X$;   it is straightforward to check that $$|\mathcal Ff(u)|^{2}=\frac1{(1+|bu_{1}+au_{2}|^{2})^{\alpha+\frac12}}\frac1{(1+|bu_{2}|^{2})^{\beta+\frac12}}. $$ We can give a lower bound of the bias term, let $m_{1}\ge1$ and $m_{2}\ge1$,
\begin{align*}
B_{f}(m_1,m_2) 
&\geq  \iint |\mathcal Ff(u)|^{2}{\bf 1}_{u_{2}>m_2}  {\rm d}u_{1} {\rm d}u_{2} + \iint  |\mathcal Ff(u)|^{2}{\bf 1}_{u_{2} <-m_1}{\bf 1}_{u_1 >m_{1} } {\rm d}u_{1} {\rm d}u_{2}\\
&\geq c_\alpha\int_{m_2}^\infty \frac{1}{(1+b^{2}u_{2}^{2})^{\beta+\frac12}} {\rm d}u_{2} + \iint  |\mathcal Ff(u)|^{2}{\bf 1}_{u_{2} <-m_1}{\bf 1}_{u_1 >m_{1}  } {\rm d}u_{1} {\rm d}u_{2}\\
&\geq \frac{C_1}{m_2^{2\beta}} + \iint|\mathcal Ff(u)|^{2}{\bf 1}_{u_{2} <-m_1}{\bf 1}_{u_1 >m_{1}  } {\rm d}u_{1} {\rm d}u_{2}
\end{align*} where  we used that $(1+b^{2}u_{2}^{2})\le(1+bu_{2})^{2}$ and where we defined $c_{\alpha} = \int_{\R}(1+z^{2})^{-(\alpha+\frac12)} {\rm d}z$ and $C_{1}={c_{\alpha}}\big/({2\beta b(1+b)^{2\beta}})$ if $m_{2}\ge1$.
Now for the second term, 
\begin{align*}
\iint |\mathcal Ff(u)|^{2}&{\bf 1}_{u_{2} <-m_1}{\bf 1}_{u_1 >m_{1}  } {\rm d}u_{1} {\rm d}u_{2}\\&= \int_{-\infty}^{-m_1}\frac{1}{(1+|bu_{2}|^{2})^{\beta+\frac12}} \left ( \int \frac{1}{(1+|bu_{1}+au_{2}|^{2})^{\alpha+\frac12}}{\bf 1}_{u_1 >m_{1}  }{\rm d}u_{1}\right) {\rm d}u_{2}\\
&=\int_{-\infty}^{-m_1}\frac{1}{(1+|bu_{2}|^{2})^{\beta+\frac12}}  \left ( \int \frac{1}{(1+z^{2})^{\alpha+\frac12}}{\bf 1}_{z>am_1+by  } {\rm d}z\right) {\rm d}u_{2}\\
&\geq \int_{-\infty}^{-m_1}\frac{1}{(1+|bu_{2}|^{2})^{\beta+\frac12}}  \left ( \int \frac{1}{(1+z^{2})^{\alpha+\frac12}}{\bf 1}_{z>0  }{\rm d}z\right) {\rm d}u_{2}\\
&= \frac{c_\alpha}2  \int_{-\infty}^{-m_1}\frac{1}{(1+|bu_{2}|^{2})^{\beta+\frac12}}   {\rm d}u_{2} \geq \frac{C_1}{2m_1^{2\beta}}
\end{align*} where we used that $am_{1}+by<0$ for all $y<-m_{1}$ as $a-b<-b^{2}<0$ and $m_{1}\ge1$. Gathering all terms we derive that 
\begin{align}\label{eq:Bf}
B_{f}(m_1,m_2)\ge  \frac{c_{\alpha}}{2b\beta(1+b)^{2\beta}}\left(\frac{1}{m_2^{2\beta}}+ \frac{1}{2m_1^{2\beta}}\right)\, .
\end{align}
Now consider the matrix $A_{a,b}$ 
which is such that $$|\mathcal Ff(A_{a,b}u)|^{2}=\frac1{(1+|u_{1}|^{2})^{\alpha+\frac12}}\frac1{(1+|u_{2}|^{2})^{\beta+\frac12}}.$$ The associated bias term can be bounded as follows 
\begin{align}
B_{f_{A_{a,b}}}(m_1,m_2) 
&\leq \iint |\mathcal Ff(A_{a,b}u)|^{2}{\bf 1}_{|u_{1}|>m_1}  {\rm d}u_{1} {\rm d}u_{2} + \iint  |\mathcal Ff(A_{a,b}u)|^{2}{\bf 1}_{|u_{2}|>m_2} {\rm d}u_{1} {\rm d}u_{2}\nonumber\\
&=c_{\beta}\int_{m_{1}}^{\infty}\frac1{(1+u_{1}^{2})^{\alpha+\frac12}}  {\rm d}u_{1} +c_{\alpha}\int_{m_{2}}^{\infty}\frac1{(1+u_{2}^{2})^{\beta+\frac12}} {\rm d}u_{2}\nonumber \\
&\le c_{\beta}\int_{m_{1}}^{\infty}\frac1{u_{1}^{2\alpha+1}}  {\rm d}u_{1} +c_{\alpha}\int_{m_{2}}^{\infty}\frac1{u_{2}^{2\beta+1}} {\rm d}u_{2} = \frac{c_{\beta}}{2\alpha}\frac1{m_{1}^{2\alpha}}+ \frac{c_{\alpha}}{2\beta}\frac1{m_{2}^{2\beta}},\label{eq:BfA}
\end{align} 
where $c_{\beta}=\int_{\R}{(1+z^{2})^{-(\beta+\frac12)}} {\rm d}z $.
Gathering \eqref{eq:Bf} and \eqref{eq:BfA} we get for $\alpha>\beta$ the announced result \eqref{eq:Bf2} if $m_{1}\ge1\vee C^{1}_{\alpha,\beta,b},\ m_{2}\ge1\vee C^{2}_{\alpha,\beta,b} $ which are positive constants $C^{j}_{\alpha,\beta, b},\ j=1,2$ ensuring that the right hand side of \eqref{eq:Bf} is larger than \eqref{eq:BfA}.

\subsection{Proof of Theorem \ref{thm:ADalpha}}

\subsubsection{Preliminaries}

First, we prove several lemmas that are useful for the proof of Theorem \ref{thm:ADalpha}.  

\begin{lemma}\label{lem:Ibra} Let $C_{\alpha, n}=1+ 4\sum_{k=1}^{n-1} \alpha(k)$. The following inequality holds
\begin{equation*}
	\mathbb{E}\left[ |\hat{\varphi}_{X,n}(u)-\varphi_{X}(u) |^2\right]\ \leq \frac{(2C_{\alpha,n}-1)}n \, .
	\end{equation*}
\end{lemma}

\begin{proof}[Proof of Lemma \ref{lem:Ibra}]
Note first that $$\mathrm{Var} (e^{i\langle u,X_1\rangle}) = \mathbb{E}(e^{i \langle u,X_1\rangle}e^{-i\langle u,X_1\rangle})-\mathbb{E}(e^{i \langle u,X_1\rangle})\mathbb{E}(e^{-i\langle u,X_1\rangle})  =1-|\varphi_X(u)|^2.$$ 
Using   the  stationarity  of the sequence $(X_i)$, we get   for the variance
	\begin{equation} \label{varcov}
	\mathbb{E}\left[ |\hat{\varphi}_{X,n}(u)-\varphi_X(u) |^2\right]
	 \leq \frac{1-|\varphi_X(u)|^2}{n}+ \frac{2}{n} \sum_{k=2}^n |\Cov (e^{i \langle u, X_1 \rangle},e^{i \langle u, X_k \rangle})|.
	\end{equation}
	We now apply Ibragimov's inequality \cite{Ibr62} : 
Let $X,Y$ be two  real-valued random variables that are almost surely bounded, and denote by $\|X\|_\infty$  the essential supremum norm of $|X|$. Then, it holds
\begin{align}\label{eq:Ibr}
|{\Cov} (X,Y)| \leq  2 \alpha( \sigma(X), \sigma(Y))     \|X \|_\infty   \| Y\|_ \infty  \, .
\end{align}
For complex-valued random variables, whose modulus are almost surely bounded, one can easily  prove that  (see for instance \cite{rio2017asymptotic}, Chapter 1, Exercice 7)
\begin{align} \label{cov}
|{\Cov} (X,Y)| \leq  4 \alpha( \sigma(X), \sigma(Y))     \|X \|_\infty   \| Y\|_ \infty  \,.
\end{align}
	Using  the inequality  \eqref{cov} and  that $|e^{i \langle u, X_k\rangle}|<1$, we get
	\begin{equation*}
	\sum_{k=2}^n |\Cov (e^{i\langle u, X_1\rangle},e^{i \langle u, X_k \rangle})| \leq 4 \sum_{k=2}^n \alpha(k-1).
	\end{equation*}
	It follows from \eqref{varcov} that,
	\begin{equation*} 
	\mathbb{E}\left[ |\hat{\varphi}_{X,n}(u)-\varphi_{X}(u) |^2\right]\ \leq  \frac{1}{n}  + \frac{8}{n} \sum_{k=1}^{n-1} \alpha(k)=\frac{(2C_{\alpha,n}-1)}n \, ,
	\end{equation*}
	and Lemma \ref{lem:Ibra} is proved.
\end{proof}

Our second lemma is a   Fuk-Nagaev inequality stated by Rio \cite{rio2017asymptotic}. We need a version of this inequality which was proposed as an exercise in  \cite{rio2017asymptotic} (Exercise 1, Chapter 6). 

\begin{lemma}\label{lem:FG}
Let $(Y_i)_{i>0}$ be a sequence of  real-valued random variables such that $\|Y_i\|_\infty \leq 1/2$,  and let  $(\alpha(n))_{n \geq 0}$ be the sequence of strong mixing coefficients of Definition \ref{def alpha}. Let $S_k = \sum_{i=1}^{k} (Y_i-\mathbb{E}(Y_{i}))$ and  $s_{n}^{2}$ be such that
\begin{equation}\label{def:sn}
s_n^2\geq   \sum_{i=1}^n \sum_{j=1}^{n}|\Cov(Y_i,Y_j) | .
\end{equation}
Then for any $\lambda \geq s_n$, it holds
\begin{equation*}\label{Fuk_nagaev}
\mathbb{P} \left (\sup_{k \in[1,n] }|S_k| \geq 4 \lambda   \right) \leq 4 \exp \left(-\frac{\lambda^2 \log 2}{2s_n^2}\right )+ 4n\lambda ^{-1} \alpha \left( \left[\frac{s_n^2}{\lambda} \right ]\right ).
\end{equation*}
\end{lemma}

\begin{proof}[Proof of Lemma \ref{lem:FG}]
We start from Theorem 6.2 Inequality (6.4) in \cite{rio2017asymptotic}, we take $r=\frac{\lambda^2}{s_n^2} $ and $Q=1$ (since  $\|Y_i-\mathbb{E}(Y_{i})\|_\infty \leq 1$). Using Rio's notations let
$$H\left(\frac{\lambda}{r}\right) = \alpha\left(\left[\frac{\lambda}{r}\right ]\right)  = \alpha\left(\left[\frac{s_n^2}{\lambda}\right ]\right).$$
We obtain that
\begin{align*}
\mathbb{P} \left(\sup_{k \in[1,n] }|S_k| \geq 4 \lambda   \right) &\leq 4\left( 1+\frac{\lambda^2}{r s_n^2}  \right) ^{-\frac{r}{2}} + 4n \lambda^{-1} \int_{0}^{H(  \lambda/r   )} Q(u) \,du \\
& \leq 4\left( 1+\frac{\lambda^2}{\frac{\lambda^2}{s_n^2}  s_n^2}  \right) ^{-\frac{\lambda^2}{2s_n^2}} + 4n \lambda^{-1} \int_{0}^{\alpha([s_n^2/\lambda])} 1 \, du \\
& \leq 4\times  2^{-\frac{\lambda^2}{2s_n^2}} + 4n \lambda^{-1} \alpha\left(\left[\frac{s_n^2}{\lambda}\right ]\right)\\
& \leq 4 \exp\left({-\frac{\lambda^2\log 2}{2s_n^2}}  \right)+ 4n \lambda^{-1} \alpha\left(\left[\frac{s_n^2}{\lambda}\right ]\right).
\end{align*}
and Lemma \ref{lem:FG} is proved.
\end{proof}

We  apply Lemma \ref{lem:FG} to $|\hat \phi_{X,n}(u)-\phi_{X}(u)|$. 

\begin{lemma}\label{lem:FG2} 
Recall that $C_{\alpha, n}= 1 + 4 \sum_{k=1}^{n-1} \alpha(k)$. 
For any  $b>0$, it holds
\begin{equation*}\label{eq:step1}
{\mathbb P}\left ( |\hat \phi_{X,n}(u)-\phi_{X}(u)| \geq \frac{b \sqrt{\log n}}{\sqrt n} \right )
\leq 8 n^{-b^2/93 C_{\alpha,n}} +
\frac{64 \sqrt{ 2n}}{b \sqrt{\log n}}
 \alpha\left(\left[\frac{2 \sqrt{ 2n} C_{\alpha,n} }{b \sqrt{\log n}}\right ]\right).
\end{equation*}
\end{lemma}

\begin{proof}[Proof of Lemma \ref{lem:FG2}]

Using \eqref{eq:84}, we infer that
\begin{align*}
{\mathbb P}\left ( |\hat \phi_{X,n}(u)-\phi_{X}(u)|>x \right ) \leq & \ \PP\left(\left|\frac{1}{n}\sum_{j=1}^{n}\left(\cos(\langle u,X_{j}\rangle )-\mbox{Re}(\phi_{X}(u))\right)\right|\ge \frac{x}{\sqrt 2}\right)\\ &+\PP\left(\left|\frac{1}{n}\sum_{j=1}^{n}\left(\sin(\langle u,X_{j}\rangle )-\mbox{Im}(\phi_{X}(u))\right)\right|\ge \frac{x}{\sqrt 2}\right) \, .
\end{align*}
Taking $\lambda=b \sqrt{n \log n}/8\sqrt 2$, one has 
\begin{align*}
{\mathbb P}\left ( |\hat \phi_{X,n}(u)-\phi_{X}(u)|> \frac{b \sqrt{\log n}}{\sqrt n} \right )\leq & \ \PP\left(\left|\frac{1}{2}\sum_{j=1}^{n}\left(\cos(\langle u,X_{j}\rangle )-\mbox{Re}(\phi_{X}(u))\right)\right|\ge 4 \lambda \right)\\ &+\PP\left(\left|\frac{1}{2}\sum_{j=1}^{n}\left(\sin(\langle u,X_{j}\rangle )-\mbox{Im}(\phi_{X}(u))\right)\right|\ge 4 \lambda \right) \, .
\end{align*}
We can now apply Lemma \ref{lem:FG} to the variables $Y_j= \cos(\langle u,X_{j}\rangle )/2$ (or $Y_j =\sin(\langle u,X_{j}\rangle )/2$), whose absolute values  are bounded by $1/2$. Note that, by \eqref{eq:Ibr}, one has 
$$
\sum_{i=1}^n \sum_{j=1}^n|\Cov(Y_i,Y_j) | \leq \frac n 4 \left (  1 + 4 \sum_{k=1}^{n-1} \alpha (k) \right ) = \frac n 4 C_{\alpha,n}  \, .
$$
Hence,  one can take 
$
s_n^2= n C_{\alpha,n} /4
$
in \eqref{def:sn}. It follows that
\begin{equation*}\label{eq:step1}
{\mathbb P}\left ( |\hat \phi_{X,n}(u)-\phi_{X}(u)|> \frac{b \sqrt{\log n}}{\sqrt n} \right )
\leq 8 \exp \left ( - \frac {b^2 \log n}{93 C_{\alpha,n}}\right ) + 
\frac{64 \sqrt{ 2n}}{b \sqrt{\log n}}
 \alpha\left(\left[\frac{2 \sqrt{ 2n} C_{\alpha,n} }{b \sqrt{\log n}}\right ]\right)\, ,
\end{equation*}
proving the lemma (the constant 93 comes from $64/\log 2 < 93$). 
\end{proof}

Using Lemmas  \ref{lem:Ibra} and \ref{lem:FG2}, we can now prove the last lemma of this subsection:
\begin{lemma}\label{lem:var2} 
Let $m=(m_{1},\ldots,m_{d})\in(0,n]^d$. For any $\kappa>0$ and any $A \in {\mathcal A}$, the following inequality holds 
\begin{multline*}
\frac{1}{(2\pi)^d}\int_{A([-m,m])}\E\big[|\tilde \phi_{X,n}(u)-\phi_{X}(u)|^{2}\big]{\rm d}u \\
\leq m_{1}\cdots m_{d}\frac{\left( (2C_{\alpha,n}+7)+(1+(\kappa+\sqrt {93 C_{\alpha,n}})\sqrt{\log n})^{2}\right)|\mathrm{det}(A)|}{\pi^d n}\\ +
\frac{64 \sqrt{ 2}n^{(2d+1)/2}}{\pi^d \sqrt {93 C_{\alpha,n}} \sqrt{\log n}}
 \alpha\left(\left[\frac{2 \sqrt{ 2n C_{\alpha,n}} }{\sqrt {93} \sqrt{\log n}}\right ]\right)  \, .\end{multline*}

\end{lemma}

\begin{proof}[Proof of Lemma \ref{lem:var2}]
 Considering the set  $\{u,|\phi_{X}(u)|< \frac{1+(\kappa+\sqrt {93 C_{\alpha,n}})\sqrt{\log n}}{\sqrt{n}}\}$, proceeding as for \eqref{eq:var1}, \eqref{eq:var2}, and using Lemma \ref{lem:Ibra} and \eqref{eq:detA}, we get 
\begin{multline*}
\frac{1}{(2\pi)^d}\int_{A([-m,m])}\E\big[|\tilde \phi_{X,n}(u)-\phi_{X}(u)|^{2}\big]{\rm d}u \\
\leq m_{1}\cdots m_{d}\frac{\left (2C_{\alpha,n}-1)+(1+(\kappa+\sqrt {93 C_{\alpha,n}})\sqrt{\log n})^{2}\right )|\mathrm{det}(A)|}{\pi^d n}\\ +\frac{1}{(2\pi)^d}\int_{A([-m,m])}\PP\Big(|\hat\phi_{X,n}(u)-\phi_{X}(u)|\geq \frac{\sqrt {93 C_{\alpha,n}} \sqrt{\log n}}{\sqrt{n}}\Big){\rm d}u\, .
\end{multline*}
This together with Lemma \ref{lem:FG2} (with $b= \sqrt{93 C_{\alpha,n}})$ yield
\begin{multline*}
\frac{1}{(2\pi)^d}\int_{A([-m,m])}\E\big[|\tilde \phi_{X,n}(u)-\phi_{X}(u)|^{2}\big]{\rm d}u \\
\leq m_{1}\cdots m_{d}\frac{\left ((2C_{\alpha,n}+7)+(1+(\kappa+\sqrt {93 C_{\alpha,n}})\sqrt{\log n})^{2}\right )|\mathrm{det}(A)|}{\pi^d n}\\ +m_{1}\cdots m_{d} |\mathrm{det}(A)|\left (
\frac{64 \sqrt{ 2n}}{\pi^d \sqrt {93 C_{\alpha,n}} \sqrt{\log n}}
 \alpha\left(\left[\frac{2 \sqrt{ 2n C_{\alpha,n}} }{\sqrt {93} \sqrt{\log n}}\right ]\right) \right) \, ,\end{multline*}
 and Lemma  \ref{lem:var2} is proved (noting that $m_{1}\cdots m_{d} |\mathrm{det}(A)|\leq n^d$ for the last term).
\end{proof}
 
 \subsubsection{End of the proof of Theorem \ref{thm:ADalpha}}
 We follow exactly the proof of Theorem \ref{thm:AD} and  only indicate the changes resulting from the application of Lemma \ref{lem:FG2} and Lemma \ref{lem:var2} instead (respectively) of Lemma \ref{lem:Hoeffding} and Lemma \ref{lem:var}.
 
 The first change is that the first term in \eqref{eq:step1bis} is now controlled via Lemma \ref{lem:var2} instead of Lemma \ref{lem:var}.
 
 The second change concerns Inequality \eqref{eq:VtoB}, which becomes (applying Lemma \eqref{lem:Ibra}) : on the set $
 \{u,|\phi_X(u)| >  n^{-1/2}\}$, 
 \begin{align*}
 \E[|\tilde{\phi}_{X,n}(u)- \phi_X(u)|^2]&\leq|\phi_{X}(u)|^{2}+ \E[|\hat{\phi}_{X,n}(u)- \phi_X(u)|^2]\\
&  \leq|\phi_{X}(u)|^{2}+\frac{(2C_{\alpha,n} -1)}n \leq 2C_{\alpha,n}|\phi_{X}(u)|^{2}.
 \end{align*}
 
 The third  change concerns Inequality \eqref{eq:VtoBH}, which becomes (applying Lemma \ref{lem:FG2} with $b= \kappa$)
\begin{multline*}
 \frac{4}{(2\pi)^d}  \E  \Big[ \hspace{-0.5cm} \int_{[-n,n]^{d}\setminus A([-m,m])} \limits\hspace{-0.5cm}  |\tilde{\phi}_{X,n}(u)- \phi_X(u)|^2 \mathbf{1}_{\{|\phi_X(u)| \leq  n^{-1/2}\}}{\rm d}u\Big]\\
   \leq \frac{4}{(2\pi)^d} \int_{u\in A([m,m])^{c} }\limits\hspace{-0.5cm}   |\phi_X(u)|^2 {\rm d}u +  \frac{2^7 n^d}{\pi^d} n^{-\kappa^2/93 C_{\alpha,n}} + 
\frac{2^{10} \sqrt{2}n^{(2d+1)/2}}{\pi^d\kappa \sqrt{\log n}}
 \alpha\left(\left[\frac{2 \sqrt{ 2n} C_{\alpha,n} }{ \kappa \sqrt{\log n}}\right ]\right).
\end{multline*}

\subsection{Proof of Theorem \ref{thm:ADtau}}

\subsubsection{Preliminaries}

We state the modifications of Lemmas \ref{lem:Ibra}-\ref{lem:FG2} for $\tau$-mixing sequences whose coefficients decrease at an exponential rate. 

\begin{lemma}\label{lem:Ibratau} 
Assume that $\tau(n) \leq K a^n$ for some $K \geq 1$ and $a \in (0,1)$.
The following inequality holds
\begin{equation*}
	\mathbb{E}\left[ |\hat{\varphi}_{X,n}(u)-\varphi_{X}(u) |^2\right]\ \leq \frac{C_1(a,K)}n 
	+   \frac{2\log(n)}{n\log(1/a)}\, ,
	\end{equation*}
with $$C_1(a,K)= 1+\frac{2}{1-a} + \frac{2\log (2K\sqrt d)}{\log(1/a)} .$$
\end{lemma}

\begin{proof}[Proof of Lemma \ref{lem:Ibratau}]

We start from \eqref{varcov}  and  use the following covariance inequalities.
Let $X,Y$ be two  real-valued random variables (with $X$ almost surely bounded) and denote by $\|X\|_\infty$  the essential supremum norm of $|X|$ . Then
\begin{align*}
|{\Cov} (X,Y)| \leq  {\mathbb E}(|X ({\mathbb E}(Y|X)-{\mathbb E}(Y))|)\leq \|X\|_\infty \tau(\sigma(X), Y) \, .
\end{align*}
Similarly, for two  complex-valued random variables $X,Y$ (with $Y=Y_1+iY_2$ and $\|X\|_\infty < \infty$), one has
 \begin{align} \label{covtau}
|{\Cov} (X,Y)| \leq  \|X\|_\infty ( \tau(\sigma(X), Y_1) +\tau(\sigma(X), Y_2))\, .
\end{align}
Now, for the variable $Y_k=\cos(\langle u, X_k \rangle)$ (or $Y_k=\sin(\langle u, X_k \rangle)$), we easily see that (for $|u|\leq n$)
$$\tau(\sigma(X_1), Y_k)\leq |u|_2 \tau(k-1) \leq \sqrt d n \tau(k-1)\, .$$
Using  the inequality  \eqref{covtau} and   that $|e^{i \langle u, X_k\rangle}|<1$, we get
	\begin{equation*}
	\sum_{k=2}^n |\Cov (e^{i\langle u, X_1\rangle},e^{i \langle u, X_k \rangle})| \leq 2 \sum_{k=2}^n \min(2\sqrt d n\tau(k-1), 1).
	\end{equation*}
It follows from \eqref{varcov} that,
	\begin{equation*} 
	\mathbb{E}\left[ |\hat{\varphi}_{X,n}(u)-\varphi_{X}(u) |^2\right]\ \leq  \frac{1}{n}  + \frac{2}{n} \sum_{k=1}^{\infty} \min(2\sqrt d n\tau(k), 1)\leq  \frac{1}{n}  + \frac{2}{n} \sum_{k=1}^{\infty} \min(2\sqrt d nK a^k, 1) \, .
	\end{equation*}
Now, 
$$
\sum_{k=1}^{\infty} \min(2\sqrt d nK a^k, 1) \leq \frac{1}{1-a} + \frac{\log (2K\sqrt d)}{\log(1/a)} + \frac{\log(n)}{\log(1/a)} \, ,
$$
completing the proof of Lemma \ref{lem:Ibratau}
\end{proof}

Our second lemma is a   Fuk-Nagaev inequality.

\begin{lemma}\label{lem:FGtau}
Let $(Y_i)_{i>0}$ be a sequence of  real-valued random variables such that $\|Y_i\|_\infty \leq 1/2$  and let  $(\tau_Y(n))_{n \geq 0}$ be the sequence of $\tau$-mixing coefficients of Definition  \ref{def tau}. Define $S_k = \sum_{i=1}^{k} (Y_i-\mathbb{E}(Y_{i}))$ and  let $s_n^2$ be as in \eqref{def:sn}.
Then for any $\lambda \geq s_n$, 
\begin{equation*}\label{Fuk_nagaev}
\mathbb{P} \left (\sup_{k \in[1,n] }|S_k| \geq 5 \lambda   \right) \leq 4 \exp \left(-\frac{\lambda^2 \log 2}{2s_n^2}\right )+ 4n\lambda ^{-1} \tau_Y \left( \left[\frac{s_n^2}{\lambda} \right ]\right ).
\end{equation*}
\end{lemma}

\begin{proof}[Proof of Lemma \ref{lem:FGtau}]
The proof is the same as that of Lemma \ref{lem:FG} by using Theorem 2 in \cite{DP04}
instead of  Theorem 6.2  in \cite{rio2017asymptotic}. 
\end{proof}

We now apply Lemma \ref{lem:FGtau} to $|\hat \phi_{X,n}(u)-\phi_{X}(u)|$.

\begin{lemma}\label{lem:FG2tau} 
Assume that $\tau(n) \leq K a^n$ for some $K\geq 1$ and $a \in (0,1)$. There exists a positive constant $C_2(a,K)$ such that,
for any  $b>0$, 
\begin{equation*}\label{eq:step1}
{\mathbb P}\left ( |\hat \phi_{X,n}(u)-\phi_{X}(u)|> \frac{b \log n}{\sqrt n} \right )
\leq 8 \exp \left ( - \frac {b^2 \log n}{578 C_2(a,K)}\right ) + 
\frac{ n^{3/2} 80\sqrt{2d} K}{b \log n} a^{[C_2(a,K)10\sqrt{2n}/b]} \, .
\end{equation*}
\end{lemma}

\begin{proof}[Proof of Lemma \ref{lem:FG2tau}]
Arguing as in Lemma \ref{lem:FG2} and taking
$\lambda=b \sqrt{n} \log n/10\sqrt 2$, one has 
\begin{align*}
{\mathbb P}\left ( |\hat \phi_{X,n}(u)-\phi_{X}(u)|> \frac{b \log n}{\sqrt n} \right )\leq & \ \PP\left(\left|\frac{1}{2}\sum_{j=1}^{n}\left(\cos(\langle u,X_{j}\rangle )-\mbox{Re}(\phi_{X}(u))\right)\right|\ge 5 \lambda \right)\\ &+\PP\left(\left|\frac{1}{2}\sum_{j=1}^{n}\left(\sin(\langle u,X_{j}\rangle )-\mbox{Im}(\phi_{X}(u))\right)\right|\ge 5 \lambda \right) \, .
\end{align*}
We can now apply Lemma \ref{lem:FG} to the variables $Y_j= \cos(\langle u,X_{j}\rangle )/2$ (or $Y_j =\sin(\langle u,X_{j}\rangle )/2$), whose absolute values  are bounded by $1/2$. 
Proceeding as in the proof of Lemma \ref{lem:Ibratau}, we get
\begin{align*}
\sum_{i=1}^n \sum_{j=1}^n|\Cov(Y_i,Y_j) | &\leq \frac n 4 \left (  1 + 2 \sum_{k=1}^{n-1} \min(\sqrt d nK a^k, 1) \right )\\
& \leq  \frac n 4 \left (  1+\frac{2}{1-a} + \frac{2\log (K\sqrt d)}{\log(1/a)} +   \frac{2\log(n)}{\log(1/a)} \right ) \\
& \leq C_2(a,K) n \log n \, .
\end{align*}
Hence,  one can take 
$
s_n^2= C_2(a,K) n \log n
$
in \eqref{def:sn}. It follows that
\begin{equation*}\label{eq:step1}
{\mathbb P}\left ( |\hat \phi_{X,n}(u)-\phi_{X}(u)|> \frac{b \log n}{\sqrt n} \right )
\leq 8 \exp \left ( - \frac {b^2 \log n}{578 C_2(a,K)}\right ) + 
\frac{ n^{3/2} 80\sqrt{2d} K}{b \log n} a^{[C_2(a,K)10\sqrt{2n}/b]} \, ,
\end{equation*}
proving the lemma (the constant 578 comes from  $400/\log 2 < 578$). 
\end{proof}

Using Lemmas  \ref{lem:Ibratau} and \ref{lem:FG2tau}, we can now prove the last lemma of this subsection.
\begin{lemma}\label{lem:var2tau} 
Let $m=(m_{1},\ldots,m_{d})\in(0,n]^d$, and assume that $\tau(n) \leq K a^n$ for some $K\geq 1$ and $a \in (0,1)$. For any $\kappa>0$ and any $A \in {\mathcal A}$, the following inequality holds 
\begin{multline*}
\frac{1}{(2\pi)^d}\int_{A([-m,m])}\E\big[|\tilde \phi_{X,n}(u)-\phi_{X}(u)|^{2}\big]{\rm d}u \\
\leq m_{1}\cdots m_{d}\frac{\left (8+C_1(a,K)+2\log n/\log (1/a)+(1+(\kappa+\sqrt {578 C_2(a,K)})\log n)^{2}\right )|\mathrm{det}(A)|}{\pi^d n}\\ +
\frac{ n^{(2d+3)/2} 80\sqrt{d} K}{\pi^d \sqrt {289 C_2(a,K)} \log n} a^{[\sqrt{C_2(a,K)}10\sqrt{2n}/\sqrt {578}]} \, .
\end{multline*}

\end{lemma}

\begin{proof}[Proof of Lemma \ref{lem:var2tau}]
 Considering the set  $\{u,|\phi_{X}(u)|< \frac{1+(\kappa+\sqrt {578 C_2(a,K)})\log n}{\sqrt{n}}\}$, proceeding as for \eqref{eq:var1}, \eqref{eq:var2}, and using Lemma \ref{lem:Ibratau}, we get 
\begin{multline*}
\frac{1}{(2\pi)^d}\int_{A([-m,m])}\E\big[|\tilde \phi_{X,n}(u)-\phi_{X}(u)|^{2}\big]{\rm d}u \\
\leq m_{1}\cdots m_{d}\frac{\left(C_1(a,K)+2\log n/\log (1/a)+(1+(\kappa+\sqrt {578 C_2(a,K)})\log n)^{2}\right )|\mathrm{det}(A)|}{\pi^d n}\\ +\frac{1}{(2\pi)^d}\int_{A([-m,m])}\PP\Big(|\hat\phi_{X,n}(u)-\phi_{X}(u)|\geq \frac{ \sqrt {578 C_2(a,K)} \log n}{\sqrt{n}}\Big){\rm d}u\, .
\end{multline*}
This together with Lemma \ref{lem:FG2tau} (with $b= \sqrt {578 C_2(a,K)})$ yield
\begin{multline*}
\frac{1}{(2\pi)^d}\int_{A([-m,m])}\E\big[|\tilde \phi_{X,n}(u)-\phi_{X}(u)|^{2}\big]{\rm d}u \\ \leq 
m_{1}\cdots m_{d}\frac{\left (8+C_1(a,K)+2\log n/\log (1/a)+(1+(\kappa+\sqrt {578 C_2(a,K)})\log n)^{2}\right)|\mathrm{det}(A)|}{\pi^d n}\\ +m_{1}\cdots m_{d}|\mathrm{det}(A)|\left ( \frac{n^{3/2} 80\sqrt{d} K}{\pi^d\sqrt {289 C_2(a,K)} \log n} a^{[\sqrt{C_2(a,K)}10\sqrt{2n}/\sqrt {578}]}
\right) \, ,\end{multline*}
 and Lemma  \ref{lem:var2tau} is proved since $m_{1}\cdots m_{d} |\mathrm{det}(A)|\leq n^d$.
\end{proof}
 
 \subsubsection{End of the proof of Theorem \ref{thm:ADtau}}
 We follow exactly the proof of Theorem \ref{thm:AD} (with $\kappa_n=1 + \kappa \log n$), we  only indicate the changes resulting from the application of Lemma \ref{lem:FG2tau} and Lemma \ref{lem:var2tau} instead (respectively) of Lemma \ref{lem:Hoeffding} and Lemma \ref{lem:var}.
 
 The first change is that the first term in \eqref{eq:step1bis} is now controlled via Lemma \ref{lem:var2tau} instead of Lemma \ref{lem:var}.
 
 The second change concerns Inequality \eqref{eq:VtoB}, which becomes (applying Lemma \ref{lem:Ibratau}) : on the set $
 \{u,|\phi_X(u)| >  n^{-1/2}\}$,
 \begin{multline*}
 \E[|\tilde{\phi}_{X,n}(u)- \phi_X(u)|^2]\leq|\phi_{X}(u)|^{2}+ \E[|\hat{\phi}_{X,n}(u)- \phi_X(u)|^2]\\
  \leq|\phi_{X}(u)|^{2}+\frac{C_1(a,K)}n 
	+   \frac{2\log(n)}{n\log(1/a)} \leq \left(1+C_1(a,K)+ \frac{2\log(n) }{\log(1/a)}\right )|\phi_{X}(u)|^{2}.
 \end{multline*}
 
 The third  change concerns Inequality \eqref{eq:VtoBH}, which becomes (applying Lemma \ref{lem:FG2tau} with $b= \kappa$)
\begin{multline*}
 \frac{4}{(2\pi)^d}  \E  \Big[ \hspace{-0.5cm} \int_{[-n,n]^{d}\setminus A([-m,m])} \limits\hspace{-0.5cm}  |\tilde{\phi}_{X,n}(u)- \phi_X(u)|^2 \mathbf{1}_{\{|\phi_X(u)| \leq  n^{-1/2}\}}{\rm d}u\Big]\\
   \leq \frac{4}{(2\pi)^d} \int_{u\in A([m,m])^{c} }\limits\hspace{-0.5cm}   |\phi_X(u)|^2 {\rm d}u +  \frac{2^7 n^d}{\pi^d} n^{-\kappa^2/578 C_2(a,K)} + 
\frac{ n^{(2d+3)/2} 2^8\sqrt{2d} 5K}{\pi^d \kappa \log n} a^{[C_2(a,K)10\sqrt{2n}/\kappa]}\, .
\end{multline*}

\medskip

\noindent{\bf Acknowledgments.} The second author was partially supported by the french ANR project ``G\'eométrie dans les donn\'ees : inf\'erence statistique \& partitionnement – GeoDSIC"

{\small
\bibliographystyle{apalike}
\bibliography{ADD}

\begin{thebibliography}{}

\bibitem[Adler and Taylor, 2009]{adler2009random}
Adler, R.~J. and Taylor, J.~E. (2009).
\newblock {\em Random fields and geometry}.
\newblock Springer Science \& Business Media.

\bibitem[Andrews, 1984]{andrews1984non}
Andrews, D.~W. (1984).
\newblock Non-strong mixing autoregressive processes.
\newblock {\em Journal of Applied Probability}, 21(4):930--934.

\bibitem[Asin and Johannes, 2017]{asin2017adaptive}
Asin, N. and Johannes, J. (2017).
\newblock Adaptive nonparametric estimation in the presence of dependence.
\newblock {\em Journal of Nonparametric Statistics}, 29(4):694--730.

\bibitem[Baudry et~al., 2012]{baudry2012slope}
Baudry, J.-P., Maugis, C., and Michel, B. (2012).
\newblock Slope heuristics: overview and implementation.
\newblock {\em Statistics and Computing}, 22:455--470.

\bibitem[Berbee, 1979]{Berbee79}
Berbee, H. C.~P. (1979).
\newblock {\em Random walks with stationary increments and renewal theory},
  volume 112 of {\em Mathematical Centre Tracts}.
\newblock Mathematisch Centrum, Amsterdam.

\bibitem[Bertin and Klutchnikoff, 2017]{bertin2017pointwise}
Bertin, K. and Klutchnikoff, N. (2017).
\newblock Pointwise adaptive estimation of the marginal density of a weakly
  dependent process.
\newblock {\em Journal of Statistical Planning and Inference}, 187:115--129.

\bibitem[Bertin et~al., 2020]{BKLP20}
Bertin, K., Klutchnikoff, N., L\'{e}on, J.~R., and Prieur, C. (2020).
\newblock Adaptive density estimation on bounded domains under mixing
  conditions.
\newblock {\em Electron. J. Stat.}, 14(1):2198--2237.

\bibitem[Birg{\'e}, 2014]{birge2014model}
Birg{\'e}, L. (2014).
\newblock Model selection for density estimation with {${\mathbb L}_2$}-loss.
\newblock {\em Probability Theory and Related Fields}, 158(3-4):533--574.

\bibitem[Bradley, 1986]{Bradley86}
Bradley, R.~C. (1986).
\newblock Basic properties of strong mixing conditions.
\newblock In {\em Dependence in probability and statistics ({O}berwolfach,
  1985)}, volume~11 of {\em Progr. Probab. Statist.}, pages 165--192.
  Birkh\"{a}user Boston, Boston, MA.

\bibitem[Chac{\'o}n and Duong, 2010]{chacon2010multivariate}
Chac{\'o}n, J.~E. and Duong, T. (2010).
\newblock Multivariate plug-in bandwidth selection with unconstrained pilot
  bandwidth matrices.
\newblock {\em Test}, 19:375--398.

\bibitem[Comte et~al., 2008]{CDT2008}
Comte, F., Dedecker, J., and Taupin, M.~L. (2008).
\newblock Adaptive density deconvolution with dependent inputs.
\newblock {\em Math. Methods Statist.}, 17(2):87--112.

\bibitem[Comte and Lacour, 2013]{comte2013anisotropic}
Comte, F. and Lacour, C. (2013).
\newblock Anisotropic adaptive kernel deconvolution.
\newblock In {\em Annales de l'IHP Probabilit{\'e}s et statistiques},
  volume~49, pages 569--609.

\bibitem[Comte and Merlev{\`e}de, 2002]{comte2002adaptive}
Comte, F. and Merlev{\`e}de, F. (2002).
\newblock Adaptive estimation of the stationary density of discrete and
  continuous time mixing processes.
\newblock {\em ESAIM: Probability and Statistics}, 6:211--238.

\bibitem[Cuny et~al., 2022]{CDM23}
Cuny, C., Dedecker, J., and Merlev{\`e}de, F. (2022).
\newblock Deviation and concentration inequalities for dynamical systems with
  subexponential decay of correlation.
\newblock {\em arXiv preprint arXiv:2201.10144}.

\bibitem[Dedecker et~al., 2007]{dedecker2007weak}
Dedecker, J., Doukhan, P., Lang, G., Le\'{o}n, J.~R., Louhichi, S., and Prieur,
  C. (2007).
\newblock {\em Weak dependence: with examples and applications}, volume 190 of
  {\em Lecture Notes in Statistics}.
\newblock Springer, New York.

\bibitem[Dedecker and Prieur, 2004]{DP04}
Dedecker, J. and Prieur, C. (2004).
\newblock Coupling for {$\tau$}-dependent sequences and applications.
\newblock {\em J. Theoret. Probab.}, 17(4):861--885.

\bibitem[Dedecker and Prieur, 2005]{dedecker2005new}
Dedecker, J. and Prieur, C. (2005).
\newblock New dependence coefficients. examples and applications to statistics.
\newblock {\em Probability Theory and Related Fields}, 132(2):203--236.

\bibitem[Donoho et~al., 1995]{donoho1995wavelet}
Donoho, D.~L., Johnstone, I.~M., Kerkyacharian, G., and Picard, D. (1995).
\newblock Wavelet shrinkage: asymptopia?
\newblock {\em Journal of the Royal Statistical Society: Series B
  (Methodological)}, 57(2):301--337.

\bibitem[Doukhan et~al., 1994]{doukhan1994functional}
Doukhan, P., Massart, P., and Rio, E. (1994).
\newblock The functional central limit theorem for strongly mixing processes.
\newblock {\em Ann. Inst. H. Poincar\'{e} Probab. Statist.}, 30(1):63--82.

\bibitem[Duval and Kappus, 2019]{duval2019adaptive}
Duval, C. and Kappus, J. (2019).
\newblock Adaptive procedure for {F}ourier estimators: application to
  deconvolution and decompounding.
\newblock {\em Electronic Journal of Statistics}, 13(2):3424--3452.

\bibitem[Gannaz and Wintenberger, 2010]{gannaz2010adaptive}
Gannaz, I. and Wintenberger, O. (2010).
\newblock Adaptive density estimation under weak dependence.
\newblock {\em ESAIM: Probability and Statistics}, 14:151--172.

\bibitem[Goldenshluger and Lepski, 2011]{goldenshluger2011bandwidth}
Goldenshluger, A. and Lepski, O. (2011).
\newblock Bandwidth selection in kernel density estimation: oracle inequalities
  and adaptive minimax optimality.
\newblock {\em The Annals of Statistics}, pages 1608--1632.

\bibitem[Goldenshluger and Lepski, 2014]{goldenshluger2014adaptive}
Goldenshluger, A. and Lepski, O. (2014).
\newblock On adaptive minimax density estimation on {$R^d$}.
\newblock {\em Probab. Theory Related Fields}, 159(3-4):479--543.

\bibitem[Hasminskii and Ibragimov, 1990]{10.1214/aos/1176347736}
Hasminskii, R. and Ibragimov, I. (1990).
\newblock {On Density Estimation in the View of Kolmogorov's Ideas in
  Approximation Theory}.
\newblock {\em The Annals of Statistics}, 18(3):999 -- 1010.

\bibitem[Ibragimov, 1962]{Ibr62}
Ibragimov, I.~A. (1962).
\newblock Some limit theorems for stationary processes.
\newblock {\em Teor. Verojatnost. i Primenen.}, 7:361--392.

\bibitem[Kerkyacharian and Picard, 2000]{MR1821645}
Kerkyacharian, G. and Picard, D. (2000).
\newblock Thresholding algorithms, maxisets and well-concentrated bases.
\newblock {\em Test}, 9(2):283--344.
\newblock With comments, and a rejoinder by the authors.

\bibitem[Lacour et~al., 2017]{lacour2017estimator}
Lacour, C., Massart, P., and Rivoirard, V. (2017).
\newblock Estimator selection: a new method with applications to kernel density
  estimation.
\newblock {\em Sankhya A}, 79(2):298--335.

\bibitem[Lerasle, 2009]{lerasle2009adaptive}
Lerasle, M. (2009).
\newblock Adaptive density estimation of stationary $\beta$-mixing and
  $\tau$-mixing processes.
\newblock {\em Mathematical Methods of statistics}, 18(1):59--83.

\bibitem[Lerasle, 2011]{lerasle2011optimal}
Lerasle, M. (2011).
\newblock Optimal model selection for density estimation of stationary data
  under various mixing conditions.
\newblock {\em The Annals of Statistics}, 39(4):1852--1877.

\bibitem[Lerasle, 2012]{lerasle2012optimal}
Lerasle, M. (2012).
\newblock Optimal model selection in density estimation.
\newblock In {\em Annales de l'IHP Probabilit{\'e}s et statistiques},
  volume~48, pages 884--908.

\bibitem[Neumann, 2000]{MR1769750}
Neumann, M.~H. (2000).
\newblock Multivariate wavelet thresholding in anisotropic function spaces.
\newblock {\em Statist. Sinica}, 10(2):399--431.

\bibitem[Rebelles, 2015]{10.1214/15-EJS986}
Rebelles, G. (2015).
\newblock {$\mathbb{L}_{p}$ adaptive estimation of an anisotropic density under
  independence hypothesis}.
\newblock {\em Electronic Journal of Statistics}, 9(1):106 -- 134.

\bibitem[Rio, 2017]{rio2017asymptotic}
Rio, E. (2017).
\newblock {\em Asymptotic {Theory} of {Weakly} {Dependent} {Random}
  {Processes}}, volume~80 of {\em Probability {Theory} and {Stochastic}
  {Modelling}}.
\newblock Springer Berlin Heidelberg, Berlin, Heidelberg.

\bibitem[Rosenblatt, 1956]{rosenblatt1956central}
Rosenblatt, M. (1956).
\newblock A central limit theorem and a strong mixing condition.
\newblock {\em Proc. Nat. Acad. Sci. U.S.A.}, 42:43--47.

\bibitem[Talagrand, 1996]{talagrand1996new}
Talagrand, M. (1996).
\newblock New concentration inequalities in product spaces.
\newblock {\em Inventiones mathematicae}, 126(3):505--563.

\bibitem[Tribouley and Viennet, 1998]{tribouley1998lp}
Tribouley, K. and Viennet, G. (1998).
\newblock Lp adaptive density estimation in a $\beta$ mixing framework.
\newblock {\em Annales de l'Institut Henri Poincar\'e (B) Probability and
  Statistics}, 34(2):179--208.

\bibitem[Varet et~al., 2019]{varet2019numerical}
Varet, S., Lacour, C., Massart, P., and Rivoirard, V. (2019).
\newblock Numerical performance of penalized comparison to overfitting for
  multivariate kernel density estimation.
\newblock {\em arXiv preprint arXiv:1902.01075}.

\bibitem[Volkonski\u{\i} and Rozanov, 1959]{volkonskii1959some}
Volkonski\u{\i}, V.~A. and Rozanov, Y.~A. (1959).
\newblock Some limit theorems for random functions. {I}.
\newblock {\em Theor. Probability Appl.}, 4:178--197.

\end{thebibliography}
}


%
\end{document}